%% file: improved.tex
\documentclass[11pt,reqno]{elsarticle}

\makeatletter
\def\ps@pprintTitle{%
 \let\@oddhead\@empty
 \let\@evenhead\@empty
 \def\@oddfoot{}%
 \let\@evenfoot\@oddfoot}
\makeatother

\usepackage{amsmath,amsthm,amscd,wrapfig,amsfonts,amssymb}
\usepackage{enumitem,comment,mathabx}
\usepackage{booktabs,multirow}

\usepackage{url,parskip}
\usepackage[top=22mm, bottom=22mm, left=22mm, right=22mm]{geometry}

\newtheorem{theorem}{Theorem}
\newtheorem{lemma}{Lemma}
\newtheorem{remark}{Remark}

\usepackage{palatino}

\newdimen\CdotAxis
\newcommand*{\CdotAux}[3]{%
  {%
    \settoheight\CdotAxis{$#2\vcenter{}$}%
    \sbox0{%
      \raisebox\CdotAxis{%
        \scalebox{#1}{%
          \raisebox{-\CdotAxis}{%
            $\mathsurround=0pt #2#3$%
          }%
        }%
      }%
    }%
    \dp0=0pt %
    \sbox2{$#2\bullet$}%
    \ifdim\ht2<\ht0 %
      \ht0=\ht2 %
    \fi
    \sbox2{$\mathsurround=0pt #2#3$}%
    \hbox to \wd2{\hss\usebox{0}\hss}%
  }%
}

\providecommand{\Lp}[1]{\ensuremath{L^{#1}\left(\mathbb{R}\right)}}
\providecommand{\C}[1]{\ensuremath{C^{#1}\left(\mathbb{R}\right)}}

\providecommand{\OpT}[1]{\ensuremath{T\left[{#1}\right]}}

\providecommand{\OpWb}[1]{\ensuremath{W_b\left[\hskip .1em{#1}\hskip .1em \right]}}
\providecommand{\OpWTb}[1]{\ensuremath{\widetilde{W}_b\left[\hskip .1em{#1}\hskip .1em\right]}}

\providecommand{\OpR}[1]{\ensuremath{R\left[\hskip .1em{#1}\hskip .1em \right]}}

\providecommand{\OpTb}[1]{\ensuremath{T_b\left[{#1}\right]}}

\providecommand{\OpS}[1]{\ensuremath{S\left[{\hskip .1em {#1}\hskip .1em} \right]}}

\providecommand{\Sr}{S(\mathbb{R})}
\providecommand{\Spr}{S'(\mathbb{R})}

\makeatletter
\g@addto@macro\normalsize{%
  \setlength\abovedisplayskip{.4em}
  \setlength\belowdisplayskip{.4em}
  \setlength\abovedisplayshortskip{.4em}
  \setlength\belowdisplayshortskip{.4em}
}

\begin{document}

\begin{frontmatter}
\begin{abstract}
\input{abstract.tex}
\end{abstract}


\title
{
Improved estimates for nonoscillatory phase functions
}

\author[jb]{James Bremer\corref{cor1}}
\ead{bremer@math.ucdavis.edu}
\author[vr]{Vladimir Rokhlin}

\cortext[cor1]{Corresponding author}

\address[jb]{Department of Mathematics, University of California, Davis}
\address[vr]{Department of Computer Science, Yale University}

\end{frontmatter}

\input{introduction.tex}

\input{preliminaries.tex}

\input{inteq.tex}

\input{overview.tex}

\input{bandlimit.tex}

\input{spectrum.tex}

\input{solution.tex}

\input{acknowledgements.tex}

\begin{section}{References}
\bibliographystyle{elsarticle-num}
\bibliography{improved}
\end{section}

\end{document}

%% file: abstract.tex
Recently, it was observed that solutions of a 
large class of highly oscillatory second order linear 
ordinary differential equations can be approximated using nonoscillatory 
phase functions.  In particular, under mild assumptions
on the coefficients and wavenumber $\lambda$ of the equation, there exists
a function whose Fourier transform decays as  $\exp(-\mu |\xi|)$
and which represents solutions of the differential equation with
accuracy on the order of $\lambda^{-1} \exp(-\mu \lambda)$.  
In this article, we establish an improved existence theorem for nonoscillatory
phase functions.    Among other things,
we show that solutions of second order linear ordinary differential
equations can be represented with accuracy  on the order
of  $\lambda^{-1} \exp(-\mu \lambda)$ using 
functions in the space $\Sr$ of rapidly decaying Schwartz functions
whose Fourier transforms are both exponentially
decaying and compactly supported.
These new observations play an important role in the analysis of a method
for the numerical solution of second order ordinary differential equations
whose running time is independent of the parameter $\lambda$.
This algorithm  will be reported at a later date.

%% file: introduction.tex
\begin{section}{Introduction} 

Given a differential equation
\begin{equation}
y''(t) + \lambda^2 q(t) y(t)  = 0\ \ \ \mbox{for all}\ \  a\leq t \leq b,
\label{introduction:original_equation}
\end{equation}
where $\lambda$ is a real number and $q:[0,1]\to\mathbb{R}$ is smooth and strictly positive,
a sufficiently smooth $\alpha:[a,b]\to\mathbb{R}$ is a phase function
for (\ref{introduction:original_equation})
if the pair of functions $u,v$ defined by the formulas
\begin{equation}
u(t) = \frac{\cos(\alpha(t)) }{\left|\alpha'(t)\right|^{1/2}}
\label{introduction:u}
\end{equation}
and
\begin{equation}
v(t) = \frac{\sin(\alpha(t)) }{\left|\alpha'(t)\right|^{1/2}}
\label{introduction:v}
\end{equation}
form a basis in the space of solutions of (\ref{introduction:original_equation}).
Phase functions have been extensively studied: they were first introduced in \cite{Kummer},
play a key role in the  theory of global transformations
of ordinary differential equations  \cite{Boruvka,Neuman}, and are
an important element in the theory of special functions
\cite{Spigler-Vianello,Goldstein-Thaler,NISTHandbook,Andrews-Askey-Roy}.

It was observed by E.E. Kummer in \cite{Kummer}  that $\alpha$ is  a phase function for 
(\ref{introduction:original_equation})  if and only if it satisfies the 
third order nonlinear differential equation
\begin{equation}
\left(\alpha'(t)\right)^2 = \lambda^2q(t) - \frac{1}{2}\frac{\alpha'''(t)}{\alpha'(t)}
+ \frac{3}{4} \left(\frac{\alpha''(t)}{\alpha'(t)}\right)^2
\label{introduction:kummers_equation}
\end{equation}
on the interval $[a,b]$.  
The presence of quotients in (\ref{introduction:kummers_equation})
is often inconvenient, and we prefer the more tractable equation
\begin{equation}
r''(t) - \frac{1}{4} \left(r'(t)\right)^2 + 4\lambda ^2 \left( \exp(r(t)) - q(t)\right) = 0 
\label{introduction:kummers_equation_logarithm}
\end{equation}
obtained from (\ref{introduction:kummers_equation}) by letting
\begin{equation}
\alpha'(t) = \lambda \exp\left(\frac{r(t)}{2}\right).
\end{equation}
Of course, if $r$ is a solution of (\ref{introduction:kummers_equation_logarithm})
then the function $\alpha$ defined by the formula
\begin{equation}
\alpha(t) = \lambda \int_a^t \exp\left(\frac{r(u)}{2}\right)\ du
\end{equation}
is a solution of (\ref{introduction:kummers_equation}).
We will refer to (\ref{introduction:kummers_equation}) as Kummer's equation
and (\ref{introduction:kummers_equation_logarithm}) as  the logarithm form
of Kummer's equation.  The form of these equations and the appearance
of $\lambda$ in them suggests that their  solutions will be oscillatory --- and most of them are.
However, there are several well-known examples of 
second order ordinary differential equations
which admit nonoscillatory phase functions.  For example, the function
\begin{equation}
\alpha(t) = \lambda \arccos(t),
\end{equation}
is a phase function for  Chebyshev's equation
\begin{equation}
y''(t) + \left(\frac{2 + t^2 + 4 \lambda^2 (1-t^2)}{4(1-t^2)^2} \right) y(t) = 0
\ \ \ \mbox{for all} \ \ -1 \leq t \leq 1.
\label{introduction:chebyshev}
\end{equation}
Its existence is the basis of many numerical algorithms, including the fast Chebyshev
transform (see, for instance, \cite{Trefethen}).
Bessel's equation
\begin{equation}
y''(t) + \left(1-\frac{\lambda^2-1/4}{t^2} \right) y(t) = 0
\ \ \ \mbox{for all}\ \ 0 < t < \infty
\label{introduction:bessel}
\end{equation}
also admits a nonoscillatory phase function, although it cannot be 
expressed in terms of elementary functions (see, for instance,
\cite{Heitman-Bremer-Rokhlin-Vioreanu}).

Exact solutions of  (\ref{introduction:kummers_equation}) 
 which are nonoscillatory need not  exist in the general case.  
However, as we show in this article, 
when the coefficient $q$ appearing in (\ref{introduction:kummers_equation}) is nonoscillatory,
there exists a nonoscillatory function $\alpha$ such that
(\ref{introduction:u}), (\ref{introduction:v})
 approximate solutions of (\ref{introduction:original_equation})
in the space $L^\infty\left(\left[a,b\right]\right)$ with accuracy
on the order of $\lambda^{-1}\exp(-\mu \lambda)$.
In order to make this statement rigorous, 
we will use the Fourier transform to quantify the notion 
of ``nonoscillatory function.''  Accordingly, we assume that the coefficient
$q$ in (\ref{introduction:original_equation}) extends to a strictly
positive function on the entire real line.  Moreover, we
define the function $x$ via the formula
\begin{equation}
x(t) = \int_a^t \sqrt{q(u)}\ du,
\end{equation}
and let  $p(x)$ be twice the Schwarzian
derivative 
of the variable $t$ with respect to the variable $x$
(see Section~\ref{section:preliminaries:schwarzian_derivative}).
We suppose that there exist constants $\Gamma$  and $\mu$ such that
\begin{equation}
\left|\widehat{p}(\xi)\right| \leq \Gamma \exp(-\mu |\xi|)
\ \ \ \mbox{for all}\ \ \xi \in\mathbb{R},
\end{equation}
 that $\lambda > \frac{2}{\mu}$, and that $\lambda > 2\Gamma$.
Then there exists a function $\delta$ in the space 
$\Sr$ of rapidly decaying
Schwartz functions (see Section~\ref{section:preliminaries:schwartz}) such that
\begin{equation}
r(t) = \log(q(t)) + \delta(x(t))
\end{equation}
closely approximates a solution of (\ref{introduction:kummers_equation_logarithm}),
the support of the Fourier transform $\widehat{\delta}$  of $\delta$
is contained in the interval $(-\sqrt{2}\lambda,\sqrt{2}\lambda)$, and
\begin{equation}
\left|\widehat{\delta}(\xi)\right| \leq  
\frac{\Gamma}{\lambda}
\exp\left(-\mu\left|\xi\right|\right)
\ \ \ \mbox{for all}\ \ \left|\xi\right| < \sqrt{2}\lambda.
\label{introduction:bound}
\end{equation}
Moreover the function $\alpha$ defined via the formula
\begin{equation}
\alpha(t) = \lambda \int_a^t \exp\left(\frac{r(u)}{2}\right)\ du
\label{introduction:alpha2}
\end{equation}
is a phase function for a second order differential equation
of the form
\begin{equation}
y''(t) + \lambda^2 \left(1 + \frac{\nu(t)}{4\lambda^2} \right) q(t)y(t)  
= 0\ \ \ \mbox{for all}\ \  a\leq t \leq b,
\label{introduction:modified_equation}
\end{equation}
where $\nu$ is an element of $\Sr$ 
whose $\Lp{\infty}$ norm is on the order of 
$\exp(-\mu\lambda)$.  
While $\alpha$ is not a phase function for the original
equation (\ref{introduction:original_equation}), 
the functions $u$, $v$ obtained by inserting
(\ref{introduction:alpha2}) into the formulas
(\ref{introduction:u}) and (\ref{introduction:v}) 
approximate solutions of (\ref{introduction:original_equation})
with accuracy on the order of $\lambda^{-1} \exp\left(-\mu\lambda\right)$.

The bound (\ref{introduction:bound}) implies that when $\delta$
is approximated using various series expansions, the 
number of terms required to represent it to a specified
precision is  independent of $\lambda$.
For instance,  the minimum number of terms in the Legendre expansion
of the restriction of $\delta$ to the interval $[a,b]$
required to achieve a specified precision is independent of $\lambda$.
Assuming that $q$ is nonoscillatory, it follows that $r$
and $\alpha$ can be represented likewise; that is,
using finite series expansions whose number of terms 
is independent of $\lambda$ (it is in this sense that they are nonoscillatory).
In other words: $O(1)$ terms are required to represent the 
function $\delta$ which enables us to approximate solutions
of (\ref{introduction:original_equation}) with accuracy
on the order of $O\left(\lambda^{-1}\exp(-\mu\lambda)\right)$.
This  is in contrast 
to superasymptotic and hyperasymptotic expansions
(see, for instance, \cite{Daalhuis,Daalhuis-Olver}) which approximate
 solutions of (\ref{introduction:original_equation}) to accuracy
on the order of $\exp(-\rho \lambda)$,
but which require $O(\lambda)$ terms in order to do so.
 Note that we avoid discussing the Fourier transforms
of $r$ and $\alpha$ because they need not decay at infinity and
so there is no assurance that 
their Fourier transforms are functions (as opposed to tempered
distributions).

The results presented in this article improve upon those of 
\cite{Heitman-Bremer-Rokhlin}, which observed that
solutions of (\ref{introduction:original_equation})
can be represented with accuracy on the order of $\exp(-\mu\lambda)$ 
using functions which are in $\Lp{2} \cap C^\infty(\mathbb{R})$ and whose Fourier transforms
decay exponentially.  
Here we show that the function $\delta$ which represents the solutions of 
(\ref{introduction:original_equation}) is an element of the space
$\Sr$ of rapidly decaying Schwartz functions 
(see Section~\ref{section:preliminaries:schwartz})
and that its Fourier transform is both exponentially decaying and
 compactly supported (so that $\delta$ is entire).
Moreover, we substantially reduce the constants
appearing in the bounds on the decay of the Fourier transform of $\delta$
and in the error  of the associated approximations of the solutions of 
(\ref{introduction:original_equation}).
Among other things, these new observations   play an important role in the analysis of 
an algorithm for constructing a nonoscillatory solution $\alpha$ of
(\ref{introduction:kummers_equation})
whose running time is independent of $\lambda$.
This algorithm allows for the numerical evaluation of solutions
of second order linear ordinary differential
equations of the form (\ref{introduction:original_equation})
using a number of operations which is independent of $\lambda$.  It will be 
reported at a later date.

The remainder of this paper is organized as follows.    
Section~\ref{section:preliminaries}
summarizes a number of well-known mathematical facts and establishes
the notation which is used throughout this article.   In Section~\ref{section:inteq}, we reformulate
Kummer's equation as a nonlinear integral equation.
The statement of the main result of this paper, Theorem~\ref{main_theorem},
is given in Section~\ref{section:overview} and
its proof  is divided among Sections~\ref{section:bandlimit},
 \ref{section:spectrum} and \ref{section:solution}.

\label{section:introduction}
\end{section}

%% file: preliminaries.tex
\begin{section}{Preliminaries}

%%%%%%%%%%%%%%%%%%%%%%%%%%%%%%%%%%%%%%%%%%%%%%%%%%%%%%%%%%%%%%%%%%%%%%%%%%%%%%%%%%%%%%%%%%
%
%  Schwartz functions and tempered distributions
%
%%%%%%%%%%%%%%%%%%%%%%%%%%%%%%%%%%%%%%%%%%%%%%%%%%%%%%%%%%%%%%%%%%%%%%%%%%%%%%%%%%%%%%%%%%%

%% \begin{subsection}{Schwartz functions and tempered distributions}
\begin{subsection}{Function spaces}

We denote by $C\left(\mathbb{R}\right)$ the set of continuous functions
$\mathbb{R} \to \mathbb{C}$.  If $f \in C\left(\mathbb{R}\right)$
and, for each $\epsilon >0$, there exists a compact set $K$ such that 
$\left|f(x)\right| < \epsilon$ for all $x \notin K$, then we say
that $f$ vanishes at infinity.  We denote the set of continuous functions
which vanish at infinity by $C_0\left(\mathbb{R}\right)$.
By $C^\infty(\mathbb{R})$ we mean the set of infinitely differentiable
functions $\mathbb{R} \to \mathbb{C}$, and $C_c^\infty\left(\mathbb{R}\right)$
is the set of compact supported functions in $C^\infty\left(\mathbb{R}\right)$.

We say that $\varphi \in \C{\infty}$ 
is a Schwartz function if $\varphi$ and all of its derivatives decay faster 
than any polynomial.  That is, if 
\begin{equation}
\sup_{t\in\mathbb{R}} |t^i \varphi^{(j)}(t)| < \infty
\end{equation}
for all pairs $i,j$ of nonnegative integers.  The set  of all Schwartz functions 
is denoted by  $\Sr$; clearly, it contains the set $C_c^\infty\left(\mathbb{R}\right)$.
We endow $\Sr$ with the topology generated by the family of seminorms
\begin{equation}
\|\varphi\|_{k} =\sum_{j=0}^k\ \sup_{t\in\mathbb{R}} \left|t^k \varphi^{(j)}(x)\right|
\ \ \ k=0,1,2,\ldots
\end{equation}
so that a sequence $\{\varphi_n\}$ of functions in $\Sr$ converges
to $\varphi$ in $\Sr$ if and only if
\begin{equation}
\lim_{n\to\infty}\|\varphi_n - \varphi\|_{k} = 0 \ \ \ \mbox{for all}\ \ k=0,1,2,\ldots.
\end{equation}

We denote the space of continuous linear functionals on $\Sr$, which are known
as tempered distributions, by $\Spr$.  We endow $\Spr$ with the weak-*
topology so that a sequence $\{\omega_n\}$ in $\Spr$ converges to $\omega \in \Spr$
if and only if 
\begin{equation}
\lim_{n\to\infty} \omega_n(\varphi) =  \omega(\varphi)
\end{equation}
for all  $\varphi \in \Sr$.  
%% We call a function $\psi \in C^\infty(\mathbb{R})$ slowly increasing
%% if for every integer $i \geq0 $ there exists an integer
%% $j \geq 0$ and a real number $C \geq 0$ such that
%% %
%% \begin{equation}
%%  \left|\psi^{(i)}(x)\right| \leq C \left(1+\left|x\right|\right)^j
%% \ \ \ \mbox{for all}\ \ x \in \mathbb{R}.
%% \end{equation}
%% %
%% In other words, an infinitely differentiable 
%% function $\psi$ is slowly increasing if $\psi$
%% and all of its derivatives have at most polynomial growth at infinity.
%% The product of an element $\varphi$ of $\Sr$ and a slowly increasing function $\psi$
%% is an element of $\Sr$ and the product of a slowly increasing
%% function and a tempered distribution is a tempered distribution.
We refer the reader to \cite{HormanderI}
for a thorough discussion of the properties of Schwartz functions
and tempered distributions.

\label{section:preliminaries:schwartz}
\end{subsection}

\begin{subsection}{The Fourier transform}
We define the Fourier transform of a function $f \in \Sr$ via the formula
\begin{equation}
\widehat{f}(\xi) = \int_{-\infty}^{\infty} \exp(-i x\xi)f(x)\ dx.
\label{preliminaries:fourier:1}
\end{equation}
The Fourier transform is an isomorphism $\Sr \to \Sr$ (meaning that it
is a continuous, invertible mapping $\Sr \to \Sr$ whose inverse is also
continuous).  
%% We define $\widehat{\omega}$ for $\omega \in \Spr$ through the formula
The formula
\begin{equation}
\left<\widehat{\omega},\varphi\right> = \left<\omega,\widehat{\varphi}\right>
%% \widehat{\omega}(\varphi) = \omega(\widehat{\varphi}),
\label{preliminaires:fourier:2}
\end{equation}
 extends the Fourier transform to an isomorphism $\Spr \to \Spr$.
The definition (\ref{preliminaires:fourier:2})
coincides with (\ref{preliminaries:fourier:1}) when $f \in \Lp{1}$.
Moreover, when $f \in \Lp{2}$,
\begin{equation}
\widehat{f}(\xi) = \lim_{R\to\infty} \int_{-R}^{R} \exp(-i x\xi)f(x)\ dx.
\label{preliminaries:fourier:3}
\end{equation}
Owing to our choice of convention for the Fourier transform,
\begin{equation}
\widehat{f*g}(\xi) = \widehat{f}(\xi) \widehat{g}(\xi)
\label{preliminaries:fourier:4}
\end{equation}
and
\begin{equation}
\widehat{f \cdot g}(\xi) = \frac{1}{2\pi} \int_{-\infty}^\infty \widehat{f}(\xi-\eta)
\widehat{g}(\eta)\ d\eta
\label{preliminaries:fourier:5}
\end{equation}
whenever $f$ and $g$ are elements of $\Lp{1}$.
Moreover, 
\begin{equation}
f(x) = \frac{1}{2\pi}\int_{-\infty}^\infty \exp(ix\xi) \widehat{f}(\xi)\ d\xi
\end{equation}
whenever $f$ and $\widehat{f}$ are elements of $\Lp{1}$.  
The observation that $f$ is an entire function when
 $\widehat{f}$ is a compactly supported distribution
 is one consequence of the well-known Paley-Wiener theorem.
See \cite{GrafakosC, Grafakos} for a thorough treatment
of the Fourier transform.

\end{subsection}

%%%%%%%%%%%%%%%%%%%%%%%%%%%%%%%%%%%%%%%%%%%%%%%%%%%%%%%%%%%%%%%%%%%%%%%%%%%%%%%%%%%%%%%%%%
%
%  Convolution exponentials
%
%%%%%%%%%%%%%%%%%%%%%%%%%%%%%%%%%%%%%%%%%%%%%%%%%%%%%%%%%%%%%%%%%%%%%%%%%%%%%%%%%%%%%%%%%%%

\begin{subsection}{Convolution exponentials}
Formally, the Fourier transform of 
\begin{equation}
\exp\left(f(x)\right)
\end{equation}
is the sum
\begin{equation}
2\pi \delta(\xi) + \psi(\xi) + \frac{\psi*\psi(\xi)}{2! (2\pi)} +
\frac{\psi*\psi*\psi(\xi)}{3! (2\pi)^2}  + \cdots,
\label{preliminaries:convolution:1}
\end{equation}
where $\psi$ denotes the Fourier transform of $f$ and  $\delta$ is the delta distribution.   
The expression (\ref{preliminaries:convolution:1}), which is referred to as
the convolution exponential of $\psi$, is typically
denoted by  $\exp^*\left[\psi\right]$.

We will not encounter the
expression (\ref{preliminaries:convolution:1}) in this article; however, we will
consider the Fourier transforms of functions of the form
\begin{equation}
\exp(f(x)) - 1 
\end{equation}
and
\begin{equation}
\exp(f(x))  - f(x) - 1.
\end{equation}
So, in analogy with the definition of $\exp_*$, we 
define $\exp^*_1\left[\psi\right]$ and  $\exp^*_2\left[\psi\right]$
for $\psi \in \Lp{1}$ via the formulas
\begin{equation}
\exp_1^*\left[\psi\right](\xi) = 
\psi(\xi) + \frac{\psi*\psi(\xi)}{2! (2\pi) } + \frac{\psi*\psi*\psi(\xi)}{3! (2\pi)^2 } + \cdots
\label{preliminaries:convolution:exp1}
\end{equation}
and
\begin{equation}
\exp_2^*\left[\psi\right](\xi) = 
\frac{\psi*\psi(\xi)}{2! (2\pi)} + \frac{\psi*\psi*\psi(\xi)}{3! (2\pi)^2} + \cdots.
\label{preliminaries:convolution:exp2}
\end{equation}
That is, $\exp_1^*\left[\psi\right]$ is obtained by truncating the leading
term of $\exp^*\left[\psi\right]$  and $\exp_2^*\left[\psi\right]$ 
is obtained by truncated the first two leading terms of 
$\exp^*\left[\psi\right]$.       
By repeatedly applying the inequality
\begin{equation}
\|f*g\|_1 \leq \|f\|_1 \|g\|_1,
\end{equation}
which can be found in \cite{Folland} (for instance), we obtain
\begin{equation}
\left\|\psi\right\|_1 + 
\left\|\frac{\psi*\psi}{2!(2\pi)}\right\|_1 +  
\left\|\frac{\psi*\psi*\psi}{3!(2\pi)^2}\right\|_1 + \cdots
\leq 
\sum_{n=1}^\infty \frac{\left\|\psi\right\|_1^n}{n! (2\pi)^{n-1}}
\leq  \left\|\psi\right\|_1  \exp\left(\frac{\|\psi\|_1}{2\pi}\right)
\label{preliminaries:convolution:ac1}
\end{equation}
and
\begin{equation}
\left\|\frac{\psi*\psi}{2!(2\pi)}\right\|_1 +  
\left\|\frac{\psi*\psi*\psi}{3!(2\pi)^2}\right\|_1 + \cdots
\leq 
 \sum_{n=2}^\infty \frac{\left\|\psi\right\|_1^n}{n! (2\pi)^{n-1}}
\leq \frac{\|\psi\|_1^2}{4\pi} \exp\left(\frac{\|\psi\|_1}{2\pi}\right),
\label{preliminaries:convolution:ac2}
\end{equation}
from which we see that the series (\ref{preliminaries:convolution:exp1})
and  (\ref{preliminaries:convolution:exp2})
 converge absolutely in $\Lp{1}$ when $\psi \in \mathbb{R}$
and therefore define $\Lp{1}$ functions.
Suppose that $f$ is the inverse Fourier transform of $\psi \in \Lp{1}$.
For each nonnegative integer $n$, we define $f_n$ by 
\begin{equation}
f_n(x) = \sum_{k=1}^n \frac{(f(x))^k}{k!}.
\end{equation}
We observe that $\{f_n\}$ converges in $\Lp{\infty}$ to $\exp(f(x))-1$.
Since
\begin{equation}
\widehat{f_n}(\xi) = 
\psi(\xi) + \frac{\psi*\psi(\xi)}{2! (2\pi) } + \cdots
 + \frac{\psi*\psi*\cdots*\psi(\xi)}{n! (2\pi)^{n-1} }
\end{equation}
for each nonnegative integer $n$, 
it follows from (\ref{preliminaries:convolution:ac1}) that
$\left\{\widehat{f_n}\right\}$ 
converges in $\Lp{1}$ to $\exp_1^*\left[\psi\right]$.  We conclude
that the Fourier transform of $\exp(f(x))-1$ is $\exp_1^*\left[\psi\right]$.
A nearly identical argument shows that the Fourier transform
of $\exp(f(x))-f(x)-1$ is $\exp_2^*\left[\psi\right]$.

\vskip 1em
\begin{theorem}
If $f \in \Sr$, then $\exp(f(x))-1$ and $\exp(f(x))-f(x)-1$
are elements of $\Sr$.
\label{preliminaries:convolution:theorem1}
\end{theorem}

\begin{proof}  
%% We denote the inverse Fourier transform of $\psi$ by $f$.
Since $f \in \Sr$, $\exp(|f(x)|)$ is bounded and 
\begin{equation}
\sup_{x\in\mathbb{R}}|x^k f(x)| < \infty 
\end{equation}
for any nonnegative integer $k$.  Consequently,
%% %% Since $\exp(f(x))$ is bounded and $f(x)$ decays faster than
%% %% any polynomial,
%
\begin{equation}
\sup_{x\in\mathbb{R}} \left| x^k \left(\exp(f(x))-1\right) \right|
\leq 
\sup_{x\in\mathbb{R}}\ \left| x^k f(x)\right| \exp\left(\left|f(x)\right|\right)
< \infty
\end{equation}
for any nonnegative integer $k$.  We conclude that $\exp(f(x))-1$ decays 
faster than any polynomial.
We observe that the $n^{th}$ derivative of $\exp(f(x))-1$ is of the form
\begin{equation}
P\left(f(x),f'(x),f''(x),\ldots,f^{(n)}(x)\right) \exp(f(x)),
\label{convolution:theorem2:1}
\end{equation}
where $P$ is a  polynomial in $n$ variables of the form
\begin{equation}
P(x_1,x_2,\ldots,x_n) = \sum_{1\leq k_1+k_2+\ldots+k_n \leq n} 
C_{k_1,k_2,\ldots,k_n} x_1^{k_1} x_2^{k_2} \ldots x_n^{k_n}.
\end{equation}
Since $f,f',f'',\ldots,f^{(n)}$ are elements of $\Sr$,
\begin{equation}
(f(x))^{k_1} (f'(x))^{k_2} \ldots \left(f^{(n)}(x)\right)^{k_n}
\end{equation}
decays  faster than any polynomial whenever $k_1, k_2,\ldots,k_n$
are nonnegative integers not all of which are $0$.  
We combine this observation with the fact that $\exp(f(x))$ is bounded in
order to conclude that the function (\ref{convolution:theorem2:1}) decays faster
than any polynomial; i.e., the $n^{th}$ derivative of $\exp(f(x))-1$ decays
faster than any polynomial.  Therefore $\exp(f(x))-1$ is in $\Sr$.
Since $\exp(f(x))-f(x)-1$ is obtained from $\exp(f(x))-1$ by subtracting
the Schwartz function $f$, it is also an element of $\Sr$.
\end{proof}

We combine Theorem~\ref{preliminaries:convolution:theorem1}
with the observation that  the Fourier transform is a continuous
mapping $\Sr \to \Sr$ in order to obtain the following theorem.
\vskip 1em
\begin{theorem}
If $\psi \in \Sr$, then $\exp_1^*\left[\psi\right]$ and  $\exp_2^*\left[\psi\right]$
are elements of $\Sr$.
\label{preliminaries:convolution:theorem2}
\end{theorem}

\label{section:preliminaries:convolution}
\end{subsection}

%%%%%%%%%%%%%%%%%%%%%%%%%%%%%%%%%%%%%%%%%%%%%%%%%%%%%%%%%%%%%%%%%%%%%%%%%%%%%%%%%%%%%%%%%%
%
%  The Helmholtz equation
%
%%%%%%%%%%%%%%%%%%%%%%%%%%%%%%%%%%%%%%%%%%%%%%%%%%%%%%%%%%%%%%%%%%%%%%%%%%%%%%%%%%%%%%%%%%%

\begin{subsection}{The constant coefficient Helmholtz equation}

Under certain conditions on the function $f$, a solution
of the inhomogeneous Helmholtz equation
\begin{equation}
y''(x) + \lambda^2 y(x) = f(x)
\ \ \ \mbox{for all} \ \ x\in\mathbb{R}
\label{preliminaries:helmholtz:eq}
\end{equation}
can be obtained via the Fourier transform. 
For instance, the following
theorem is a special case of a more general one which can be found in \cite{HormanderII}.

\vskip 1em
\begin{theorem}
Suppose that $f \in \Lp{1} \cap C(\mathbb{R})$, and that
 $\lambda$ is a positive real number.
Then the function $g$ defined by the formula
\begin{equation}
g(x) = \frac{1}{2\lambda} \int_{-\infty}^\infty\sin \left(\lambda \left|x-y\right|\right) f(y)\ dy
\label{preliminaries:helmholtz:g}
\end{equation}
is twice continuously differentiable, 
\begin{equation}
g''(x) + \lambda^2 g(x) = f(x) \ \ \mbox{for all} \ x\in\mathbb{R},
\label{preliminaries:helmholtz:diffeq}
\end{equation}
and 
\begin{equation}
\widehat{g}(\xi) = 
\frac{\widehat{f}(\xi)}{\lambda^2 - \xi^2}.
\label{preliminaries:helmholtz:tempdist}
\end{equation}
\label{preliminaries:helmholtz:theorem1}
\end{theorem}
We interpret the Fourier transform (\ref{preliminaries:helmholtz:tempdist})
of $g$  as a tempered distribution
defined via principal value integrals; 
that is to say that for all $\varphi \in \Sr$,
\begin{equation}
\left<\frac{\widehat{f}(\xi)}{\lambda^2-\xi^2}, \varphi \right>
=
\frac{1}{2\lambda}
\left(
\lim_{\epsilon \to 0}
\int_{|\xi-\lambda| > \epsilon} 
\frac{\widehat{f}(\xi) \varphi(\xi)}{\lambda-\xi}\ d\xi
-
\lim_{\epsilon \to 0}
\int_{|\xi+\lambda| > \epsilon} 
\frac{\widehat{f}(\xi)\varphi(\xi)}{\lambda+\xi}\ d\xi
\right).
\end{equation}
We also note that the requirement that $f$ is continuous ensures
that (\ref{preliminaries:helmholtz:diffeq}) holds for all $x\in\mathbb{R}$;
without such an assumption on $f$, we are only guaranteed that
 (\ref{preliminaries:helmholtz:diffeq}) holds
for almost all $x \in \mathbb{R}$.

When $f \in \Lp{2}$ and $\lambda$ is real-valued,
 the integral (\ref{preliminaries:helmholtz:g}) defining
the function $g$ is not necessarily absolutely convergent and the expression
\begin{equation}
\frac{\widehat{f}(\xi)}{\lambda^2-\xi^2},
\label{preliminaries:helmholtz:fhat}
\end{equation}
which is formally the Fourier transform of $g$, need not
define a tempered distribution.    If, however,  the  support
of $\widehat{f}$ is contained in
$(-\lambda,\lambda)$,  then  (\ref{preliminaries:helmholtz:fhat})
is a compactly supported element of $\Lp{1}$. 
In this case, we define $g$ through the formula
\begin{equation}
g(x) = \frac{1}{2\pi} \int_{-\infty}^\infty \exp(ix\xi) \frac{\widehat{f}(\xi)}{\lambda^2-\xi^2}\ d\xi.
\label{preliminaries:helmholtz:gft}
\end{equation}
Since $g$ is the inverse Fourier transform of a compactly supported function,
it is entire.
Moreover, the Fourier transform of 
\begin{equation}
g''(x) + \lambda^2 g(x)
\end{equation}
is
\begin{equation}
-\xi^2 \frac{\widehat{f}(\xi)}{\lambda^2-\xi^2}
+
\lambda^2 
 \frac{\widehat{f}(\xi)}{\lambda^2-\xi^2}
= \widehat{f}(\xi),
\end{equation}
from which we conclude that 
\begin{equation}
g''(x) + \lambda^2 g(x) = f(x)
\label{preliminaries:helmholtz:1}
\end{equation}
almost everywhere.  Since both $g$ and $f$ are both entire,
(\ref{preliminaries:helmholtz:1}) in fact holds for all $x\in \mathbb{R}$.
We record these observations as follows.

\vskip 1em
\begin{theorem}
Suppose that $f\in \Lp{2}$, that $\lambda$ is a positive real number,
and that the support of $\widehat{f}$ is
contained in  the interval $(-\lambda,\lambda)$.  Then the 
function $g$ defined via the formula
\begin{equation}
g(x) = \frac{1}{2\pi} \int_{-\infty}^\infty \exp(ix\xi) \frac{\widehat{f}(\xi)}{\lambda^2-\xi^2}\ d\xi
\end{equation}
is entire, 
\begin{equation}
g''(x) + \lambda^2 g(x) = f(x) \ \ \mbox{for all} \ x\in\mathbb{R},
\end{equation}
and 
\begin{equation}
\widehat{g}(\xi) = 
\frac{\widehat{f}(\xi)}{\lambda^2 - \xi^2}.
\end{equation}
\label{preliminaries:helmholtz:theorem2}
\end{theorem}

The following variant of Theorem~\ref{preliminaries:helmholtz:theorem1}
can be found in \cite{Coddington-Levinson}.
\vskip 1em
\begin{theorem}
Suppose that $f$ is continuous on the interval $[a,b]$,
and that $\lambda$ is a positive real number.
Suppose also that $y:[a,b]\to\mathbb{C}$ is twice
continuously differentiable, and that
\begin{equation}
y''(x) + \lambda^2 y(x) = f(x)
\ \ \ \mbox{for all}\ \ a \leq x \leq b.
\end{equation}
Then
\begin{equation}
y(x) = y(a) + y'(a) (x-a) + \frac{1}{\lambda}
\int_a^x 
\sin\left(\lambda\left(x-u\right)\right) f(u)\ du
\ \ \ \ \mbox{for all} \ \ a \leq x \leq b.
\end{equation}
\label{preliminaries:helmholtz:theorem3}
\end{theorem}

\label{section:preliminaries:helmholtz}
\end{subsection}

%%%%%%%%%%%%%%%%%%%%%%%%%%%%%%%%%%%%%%%%%%%%%%%%%%%%%%%%%
%
%
%
%%%%%%%%%%%%%%%%%%%%%%%%%%%%%%%%%%%%%%%%%%%%%%%%%%%%%%%%%%%%%
\begin{subsection}{Modified Bessel functions}
The modified Bessel function $K_\nu(t)$ of the first kind of order $\nu$
is defined for $t \in \mathbb{R}$ and $\nu \in \mathbb{C}$
by the formula
\begin{equation}
K_\nu(t) = \int_0^\infty \exp\left(-t\cosh\left(t\right)\right)\cosh(\nu t)\ dt.
\end{equation}
The following bound on the ratio of $K_{\nu+1}$ to $K_\nu$ can be
found in \cite{Segura}.
\vskip 1em
\begin{theorem}
Suppose that $t>0$ and $\nu >0$ are real numbers.  Then
\begin{equation}
\frac{K_{\nu+1}(t)}{K_{\nu}(t)} < \frac{\nu + \sqrt{\nu^2+t^2}}{t} 
\leq \frac{2\nu}{t}+1.
\end{equation}
\label{preliminaries:bessel:theorem1}
\end{theorem}

\label{section:preliminaries:bessel}
\end{subsection}

%%%%%%%%%%%%%%%%%%%%%%%%%%%%%%%%%%%%%%%%%%%%%%%%%
%
% Binomial theorem
%
%%%%%%%%%%%%%%%%%%%%%%%%%%%%%%%%%%%%%%%%%%%%%%%

\begin{subsection}{The binomial theorem}

A proof of the following can be found in \cite{Rudin}, as well as 
many other sources.

\vskip 1em
\begin{theorem}
Suppose that $r$ is a real number,  and that
 $y$ is a  real number such that $|y| < 1$.  Then
\begin{equation}
(1+y)^r = \sum_{k=0}^\infty  \frac{\Gamma(r+1)}{\Gamma(k+1)\Gamma(r-k+1)} y^k.
\end{equation}
\label{preliminaries:binomial_theorem}
\end{theorem}
\label{section:preliminaries:binomial}
\end{subsection}

\begin{subsection}{Fr\'echet derivatives and the contraction mapping principle}

Given Banach spaces $X$, $Y$ and a mapping $f:X \to Y$ between them, we say
that $f$ is Fr\'echet differentiable at $x \in X$ if there exists a 
bounded linear operator  
$X \to Y$, denoted by $f_x'$, such that
\begin{equation}
\lim_{h \to 0} \frac{\left\|f(x+h)-f(x)-f_x'\left[h\right]\right\|}{\|h\|} = 0.
\end{equation}

\vskip 1em
\begin{theorem}
Suppose that $X$ and $Y$ are a Banach spaces and that $f:X \to Y$ is Fr\'echet differentiable
at every point of $X$.  Suppose also that $D$ is a convex subset of $X$, and
that there exists a real number $M > 0$ such that
\begin{equation}
\|f'_x\| \leq M
\label{preliminaries:frechet:1}
\end{equation}
for all $x \in D$.   Then
\begin{equation}
\|f(x) - f(y)\| \leq M \|x-y\|
\end{equation}
for all $x$ and $y$ in $D$.
\label{preliminaries:frechet:mvt}
\end{theorem}

Suppose that $f:X \to X$ is a mapping of the Banach space $X$ into itself.
We say that $f$ is contractive on a subset $D$ of $X$  if 
there exists a real number $0 < \alpha < 1$ such that 
\begin{equation}
\|f(x)-f(y)\| \leq \alpha \|x-y\|
\end{equation}
for all $x,y \in D$.
Moreover, we say that  $\{x_n\}_{n=0}^\infty$ is a sequence
of fixed point iterates for $f$ if $x_{n+1} = f(x_n)$ for all $n \geq 0$.

Theorem~\ref{preliminaries:frechet:mvt} is often used to show
that a mapping is contractive so that the 
following result can be applied.

\vskip 1em
\begin{theorem}{(The Contraction Mapping Principle)}
Suppose that $D$ is a closed subset of a Banach space $X$.
Suppose also that $f:X \to X$ is contractive on $D$ and
$f(D) \subset D$.
Then the equation
\begin{equation}
x = f(x)
\label{frechet:1}
\end{equation}
has a unique solution $\sigma^* \in D$.    Moreover, any sequence of
 fixed point iterates for the function $f$
which contains an element in $D$ 
 converges to $\sigma^*$.
\label{preliminaries:frechet:cmp}
\end{theorem}

A discussion of Fr\'echet derivatives and 
proofs of Theorems~\ref{preliminaries:frechet:mvt} and
\ref{preliminaries:frechet:cmp} can be found in, for instance, \cite{ZeidlerI}.

\label{section:preliminaries:frechet}
\end{subsection}

%%%%%%%%%%%%%%%%%%%%%%%%%%%%%%%%%%%%%%%%%%%%%%%%%%%%%%%%%%%%%%%%%%%%%%%%%%%%%%%%%%%%%%%%%%%%%%%
%
%  Schwartzian derivative
%
%%%%%%%%%%%%%%%%%%%%%%%%%%%%%%%%%%%%%%%%%%%%%%%%%%%%%%%%%%%%%%%%%%%%%%%%%%%%%%%%%%%%%%
\begin{subsection}{Schwarzian derivatives}
The Schwarzian derivative of a smooth function $f: \mathbb{R} \to \mathbb{R}$ 
is
\begin{equation}
\{f,t\} = \frac{f'''(t)}{f'(t)} - \frac{3}{2} \left(\frac{f''(t)}{f'(t)}\right)^2.
\end{equation}
If the function $x(t)$ is a diffeomorphism of the real line (that is, a smooth,
invertible mapping $\mathbb{R} \to \mathbb{R}$),
then the Schwarzian derivative of $x(t)$
can is related to the Schwarzian derivative of its inverse $t(x)$;  in particular,
\begin{equation}
\{x,t\} = - \left(\frac{dx}{dt}\right)^2\{t,x\}.
\label{preliminaries:Schwarzian_derivative:change_of_vars}
\end{equation}
The identity~(\ref{preliminaries:Schwarzian_derivative:change_of_vars})
can be found, for instance, in Section 1.13 of \cite{NISTHandbook}.

\label{section:preliminaries:schwarzian_derivative}
\end{subsection}

%%%%%%%%%%%%%%%%%%%%%%%%%%%%%%%%%%%%%%%%%%%%%%%%%%%%%%%%%%%%%%%%%%%%%%%%%%%%%%%%%%%%%
%
%  BUMP FUNCTION
%
%%%%%%%%%%%%%%%%%%%%%%%%%%%%%%%%%%%%%%%%%%%%%%%%%%%%%%%%%%%%%%%%%%%%%%%%%%%%%%%%%%%%%%

\begin{subsection}{A bump function}

It is well known that the function $\varphi$ defined by the formula
\begin{equation}
\varphi(\xi) = 
\left(\int_{-1}^{1} \exp\left(\frac{1}{u^2-1}\right)\ du\right)^{-1}
\int_{-\infty}^\xi \exp\left(\frac{1}{u^2-1}\right) \chi_{(-1,1)}(u) \ du
\end{equation}
is an element of $C^\infty\left(\mathbb{R}\right)$ such that
$\varphi(\xi) = 0$ for all $\xi < -1$, $\varphi(\xi) = 1$
for all $\xi > 1$ and $0 \leq \varphi(\xi)\leq 1$ for  all $\xi \in\mathbb{R}$
(see, for instance, \cite{Folland}).  
We suppose that $\lambda$ is a positive real number and 
define the bump function $\widehat{b}$ via the formula
\begin{equation}
\widehat{b}(\xi) = 
\frac{1}{2} 
\left(\varphi\left(\frac{\xi+c}{\alpha}\right)-
\varphi\left(\frac{\xi-c}{\alpha}\right)  \right),
\label{preliminaries:bump:b1}
\end{equation}
where
\begin{equation}
c = \frac{\sqrt{2}\lambda+ \lambda}{2} 
\label{preliminaries:bump:c0}
\end{equation}
and
\begin{equation}
\alpha
= \frac{c- \lambda}{4}
=
\frac{\sqrt{2} \lambda-  \lambda}{4}.
\label{preliminaries:bump:alpha0}
\end{equation}
We observe that $\widehat{b}$ is an element of $C_c^\infty\left(\mathbb{R}\right)$,
the support of $\widehat{b}$ is contained in $(-\sqrt{2}\lambda,\sqrt{2}\lambda)$,
$0 \leq \widehat{b}(\xi) \leq 1$ for all $\xi\in\mathbb{R}$,  and 
$\widehat{b}(\xi) = 1$ for all $|\xi| \leq \lambda$.  
Since $\widehat{b}$ is an element of $C_c^\infty\left(\mathbb{R}\right)$, 
its inverse Fourier transform $b$ is an entire function and an element of  $\Sr$.
\label{section:preliminaries:bump}
\end{subsection}

%%%%%%%%%%%%%%%%%%%%%%%%%%%%%%%%%%%%%%%%%%%%%%%%%%%%%%%%%%%%%%%%%%%%%%%%%%%%%%%%%%%%%%%%%%%%%%%
%
%  Liouville-Green transform
%
%%%%%%%%%%%%%%%%%%%%%%%%%%%%%%%%%%%%%%%%%%%%%%%%%%%%%%%%%%%%%%%%%%%%%%%%%%%%%%%%%%%%%%
\begin{subsection}{The Liouville-Green transform}

The Liouville-Green transform is a well-known tool for analyzing the variable
coefficient Helmholtz equation 
\begin{equation}
y''(t) + \lambda^2 q(t) y(t) = f(t).
\end{equation}
The following can be found in Chapter 2 of \cite{Fedoryuk} (for instance).

\vskip 1em
\begin{theorem}
Suppose that $q:[a,b] \to \mathbb{R}$ is twice continuously differentiable and strictly
positive, that $f:[a,b] \to\mathbb{C}$ is continuous,
that the function $x$ is defined by the formula
\begin{equation}
x(t) = \int_a^t \sqrt{q(u)}\ du,
\label{preliminaries:liouville:1}
\end{equation}
and that the function $p$ is defined by the formula
\begin{equation}
%% p(x) = 2\{t,x\},
p(t) = \frac{1}{q(t)} \left(\frac{5}{4}\left(\frac{q'(t)}{q(t)}\right)^2 - \frac{q''(t)}{q(t)}
\right).
\end{equation}
Suppose also that $y:[a,b]\to\mathbb{R}$ is twice continuously differentiable, 
and that
\begin{equation}
y''(t) + \lambda^2 q(t) y(t) = f(t) \ \ \ \mbox{for all}\ \ a \leq t \leq b.
\end{equation}
Then the inverse $t(x)$ of the function $x(t)$ is continuously
differentiable, and 
the function $\varphi:[0,x(b)]\to\mathbb{R}$ defined by the formula
\begin{equation}
\varphi(x) = \left(q(t(x))\right)^{1/4} y(t(x))
\end{equation}
is the unique solution of the initial value problem
\begin{equation}
\left\{
\begin{aligned}
\varphi''(x) + \lambda^2 \varphi(x) &= 
(q(x))^{-3/4}f(x) - \frac{1}{4}p(x)\varphi(x)  \ \ \ \mbox{for all}\ \ 0 \leq x \leq x(b)\\
\varphi(0) &= (q(a))^{1/4}y(a)\\
\varphi'(0) &= q'(a) (q(a))^{-5/4} y(a)+
(q(a))^{-1/4} y'(a).
\end{aligned}
\right.
\label{preliminaries:liouville:3}
\end{equation}
\vskip 1em

\begin{remark}
We observe that due to Formula~(\ref{preliminaries:Schwarzian_derivative:change_of_vars}),
the function $p(x)$ appearing in (\ref{preliminaries:liouville:3})
is twice the Schwarzian derivative of the inverse $t(x)$ of the function $x(t)$
defined in  (\ref{preliminaries:liouville:1}).  That is, $p(x) = 2\{t,x\}$.
\end{remark}
%
%% Then the function $\varphi$ defined via the formula
%% %
%% \begin{equation}
%% \varphi(x) = 
%% \end{equation}
%
%% Suppose also  that  $\varphi$ is defined via the formula
%% %
%% \begin{equation}
%% \varphi(t) = \left(q(t)\right)^{1/4} y(t),
%% \end{equation}
%% %
%% Then $x(t)$ has an inverse $t(x)$ defined on the interval $[0,x(b)]$,
%% and  the function $\varphi(x) = \varphi(t(x))$ is the unique solution of 
%% the boundary value problem
%% %
%% \begin{equation}
%% \left\{
%% \begin{aligned}
%% \varphi''(x) + \lambda^2 \varphi(x) &= (q(x))^{-3/4} f(x)
%% -\frac{1}{4} p(x) \varphi(x) 
%% \ \ \  \mbox{for all}\ \ 0 \leq x \leq x(b)\\
%% \varphi(0) &= (q(a))^{1/4} y(a)\\
%% \varphi'(0)&=  
%% y'(a) (q(a))^{-1/4} + \frac{1}{4} y(a)q'(a)\left(q(a)\right)^{-5/4},
%% \end{aligned}
%% \right.
%% \end{equation}
%% %
%
\label{preliminaries:liouville-green:theorem1}
\end{theorem}

\label{section:preliminaries:liouville-green}
\end{subsection}

%%%%%%%%%%%%%%%%%%%%%%%%%%%%%%%%%%%%%%%%%%%%%%%%%%%%%%%%%%%%%%%%%%%%%%%%%%%%%%%%%%%%%%%%%%%%%%%
%
%  GRONWALL
%
%%%%%%%%%%%%%%%%%%%%%%%%%%%%%%%%%%%%%%%%%%%%%%%%%%%%%%%%%%%%%%%%%%%%%%%%%%%%%%%%%%%%%%
\begin{subsection}{Gronwall's inequality}
The following well-known inequality can be found in, for instance,  \cite{Bellman}.
\vskip 1em
\begin{theorem}
Suppose that $f$ and $g$ are continuous functions on the interval $[a,b]$ such that
\begin{equation}
f(t) \geq 0 \ \ \mbox{and}\ \ g(t) \geq 0 \ \ \ \mbox{for all}\ \ a \leq t \leq b.
\end{equation}
%% %
Suppose further that there exists a real number $C>0$ such that
\begin{equation}
f(t) \leq C + \int_a^t f(s)g(s)\ ds \ \ \ \mbox{for all}\ \ a \leq t \leq b.
\end{equation}
Then
\begin{equation}
f(t) \leq C \exp\left(\int_a^t g(s)\ ds\right)\ \ \ \mbox{for all}\ \ a \leq t \leq b.
\end{equation}
\label{preliminaries:gronwall:theorem1}
\end{theorem}
\label{section:preliminaries:gronwall}
\end{subsection}

\label{section:preliminaries}
\end{section}

%% file: inteq.tex
 \begin{section}{Integral equation formulation}

In this section, we reformulate Kummer's equation
\begin{equation}
\left(\alpha'(t)\right)^2 = \lambda^2q(t) - \frac{1}{2}\frac{\alpha'''(t)}{\alpha'(t)}
+ \frac{3}{4} \left(\frac{\alpha''(t)}{\alpha'(t)}\right)^2
\label{inteq:kummers_equation}
\end{equation}
as a nonlinear integral equation.  We 
assume that the function $q$ has been extended to the real line
and we seek a function $\alpha$ which satisfies (\ref{inteq:kummers_equation})
on the real line.

By letting
\begin{equation}
\left(\alpha'(t)\right)^2 = \lambda^2 \exp(r(t))
\end{equation}
in (\ref{inteq:kummers_equation}), we obtain  the equation
\begin{equation}
r''(t) - \frac{1}{4}\left(r'(t)\right)^2 + 4 \lambda^2\left( \exp(r(t)) - q(t)\right) = 0
\ \ \ \mbox{for all}\ \ \ t\in\mathbb{R}.
\label{inteq:kummer_logarithm_form}
\end{equation}
We next take $r$ to be of the form
\begin{equation}
r(t) = \log(q(t)) + \delta(t),
\end{equation}
which results in 
\begin{equation}
\delta''(t) - \frac{1}{2}\frac{q'(t)}{q(t)}\delta'(t) 
-\frac{1}{4}\left(\delta'(t)\right)^2
+ 4\lambda^2 
q(t) \left(\exp(\delta(t))-1\right) 
=
q(t) p(t),
\ \ \ \mbox{for all}\ \ \ t\in\mathbb{R},
\label{inteq:delta_equation}
\end{equation}
where $p$ is defined by the formula
\begin{equation}
p(t) = \frac{1}{q(t)} \left(
\frac{5}{4} \left(\frac{q'(t)}{q(t)}\right)^2
-\frac{q''(t)}{q(t)} \right).
\label{inteq:definition_of_p}
\end{equation}
Expanding the exponential in a power series and rearranging terms yields
the equation
\begin{equation}
\delta''(t) - \frac{1}{2}\frac{q'(t)}{q(t)}\delta'(t) + 4 \lambda^2 q(t) \delta(t)
- \frac{1}{4} \left(\delta'(t)\right)^2 + 4\lambda^2 q(t)
 \left(\frac{\left(\delta(t)\right)^2}{2} + 
\frac{\left(\delta(t)\right)^3}{3!} + \cdots \right) 
= q(t) p(t).
\label{inteq:differential}
\end{equation}
 Applying the change of variables
\begin{equation}
x(t) = \int_a^t \sqrt{q(u)}\ du
\label{inteq:changeofvars}
\end{equation}
transforms (\ref{inteq:differential}) into
\begin{equation}
\delta''(x) + 4\lambda^2 \delta(x) 
-\frac{1}{4} \left(\delta'(x)\right)^2
+ 4\lambda^2 \left(\frac{\left(\delta(x)\right)^2}{2} + 
\frac{\left(\delta(x)\right)^3}{3!} + \cdots \right) = p(x)
\ \ \ \mbox{for all}\ \ \ x \in \mathbb{R}.
\label{inteq:differential_changeofvars}
\end{equation}

At first glance, the relationship between
the function $p(x)$ appearing in (\ref{inteq:differential_changeofvars})
and the coefficient $q(t)$ in the ordinary differential equation
(\ref{introduction:original_equation}) is complex.
However,  the function $p(t)$ defined via (\ref{inteq:definition_of_p}) is related
to the Schwarzian derivative  (see Section~\ref{section:preliminaries:schwarzian_derivative}) 
of the function $x(t)$ defined in (\ref{inteq:changeofvars}) via
the formula
\begin{equation}
p(t) = - \frac{2}{q(t)} \left\{x,t\right\}
=
-2 \left(\frac{dt}{dx}\right)^2 \left\{x,t\right\}.
\label{inteq:pschwarz}
\end{equation}
It follows from (\ref{inteq:pschwarz}) and 
Formula~(\ref{preliminaries:Schwarzian_derivative:change_of_vars}) 
in Section~\ref{section:preliminaries:schwarzian_derivative}
that 
\begin{equation}
p(x) = 2 \left\{t,x\right\}.
\label{inteq:pschwarz2}
\end{equation}
That is to say: $p$, when viewed as a function of $x$, is simply
twice the Schwarzian derivative of $t$ with respect to $x$.

It is also notable that the part of (\ref{inteq:differential_changeofvars}) 
which is linear in $\delta$ is the constant coefficient  Helmholtz equation.  
This suggests that  we form an  integral equation for (\ref{inteq:differential_changeofvars})
using a Green's function for the Helmholtz equation.
To that end, we define the linear integral operator $T$  
for functions $f \in \Lp{1} \cap C\left(\mathbb{R}\right)$
via the formula
\begin{equation}
\begin{aligned}
\OpT{f}(x) &= 
\frac{1}{4\lambda}
\int_{-\infty}^\infty
\sin\left(2 \lambda \left|x-y\right|\right) f(y)\ dy.
\end{aligned}
\label{inteq:definition_of_T}
\end{equation}
We extend the domain of $T$ to include functions $f \in \Lp{2}$
whose Fourier transforms have support in 
the interval $(-2\lambda,2\lambda)$ through the formula
\begin{equation}
\OpT{f}(x) = 
\frac{1}{2\pi} 
\int_{-\infty}^\infty 
\exp(ix\xi) \frac{\widehat{f}(\xi)}{4\lambda^2-\xi^2}
\ d\xi.
\label{inteq:definition_of_T2}
\end{equation}
Introducing the representation 
\begin{equation}
\delta(x) = T\left[\sigma\right](x)
\label{inteq:delta}
\end{equation}
into (\ref{inteq:differential_changeofvars}) yields the nonlinear integral
equation
\begin{equation}
\sigma(x)  =
\OpS{\OpT{\sigma}}(x) + p(x)
\ \ \ \mbox{for all}\ \ x\in\mathbb{R},
\label{inteq:integral_equation}
\end{equation}
where $S$ is the nonlinear differential operator defined by the formula
\begin{equation}
\begin{aligned}
\OpS{f}(x) 
&= 
\frac{\left(f'(x)\right)^2}{4}
- 4 \lambda^2 \left(\frac{\left(f(x)\right)^2}{2!} + 
\frac{\left(f(x)\right)^3}{3!} + \frac{\left(f(x)\right)^4}{4!}+ \cdots \right).
\end{aligned}
\label{inteq:definition_of_S}
\end{equation}
According to Theorems~\ref{preliminaries:helmholtz:theorem1}
and \ref{preliminaries:helmholtz:theorem2},
if $\sigma$ is a solution of the integral
equation (\ref{inteq:integral_equation})
and either $\sigma \in \Lp{1} \cap C\left(\mathbb{R}\right)$
or  $\sigma \in \Lp{2}$ and the support of $\widehat{\sigma}$ 
is contained in $(-2\lambda,2\lambda)$, then
the function $\delta$ defined via formula
(\ref{inteq:delta}) is a solution of
(\ref{inteq:differential_changeofvars}).   Moreover,
the function $r$ defined via the formula
\begin{equation}
r(t) = \log(q(t)) + \delta(x(t))
\end{equation}
is a solution of (\ref{inteq:kummer_logarithm_form}),
and
\begin{equation}
\alpha(t) = \lambda \int_a^t \exp\left(\frac{r(u)}{2}\right)\ du
\end{equation}
is a phase function for (\ref{introduction:original_equation}).

\label{section:inteq}
\end{section}

%% file: overview.tex
\begin{section}{Existence of nonoscillatory phase functions}

The nonlinear integral equation (\ref{inteq:integral_equation}) is not solvable for arbitrary $p$.
However, when the Fourier transform of the function $p$ decays exponentially,
 there exists a function $\sigma$ whose Fourier transform is compactly supported
and a function $\nu$ of magnitude on the order of $\exp(-\rho\lambda)$, where
$\rho$ is a real constant, such that
\begin{equation}
\sigma(x)  =
\OpS{\OpT{\sigma}}(x) + p(x) + \nu(x)
\ \ \ \mbox{for all}\ \ x\in\mathbb{R}.
\label{overview:inteq}
\end{equation}
The following theorem, which is the principal result of this article, makes
these statements precise.  Its proof is given in Sections~\ref{section:bandlimit},
\ref{section:spectrum} and \ref{section:solution}.

%%%%%%%%%%%%%%%%%%%%%%%%%%%%%%%%%%%%%%%%%%%%%%%%%
%
%  MAIN THEOREM
%
%%%%%%%%%%%%%%%%%%%%%%%%%%%%%%%%%%%%%%%%%%%%%%%%%
\vskip 1em
\begin{theorem}
Suppose that  $q \in C^\infty\left(\mathbb{R}\right)$ is strictly positive,
that $x(t)$ is defined by the formula
\begin{equation}
x(t) = \int_0^t \sqrt{q(u)}\ du,
\label{main_theorem:x}
\end{equation}
and that the function $p$ defined via the formula
\begin{equation}
p(x) = 2\{t,x\}
\end{equation}
is an element of $\Sr$. Suppose furthermore that there exist positive real numbers
$\lambda$, $\Gamma$ and $\mu$ such that
\begin{equation}
  \lambda >  2\max\left\{\frac{1}{\mu},\Gamma\right\}
\label{main_theorem:lambda}
\end{equation}
and
\begin{equation}
\left|\widehat{p}(\xi)\right| \leq 
\Gamma \exp\left(-\mu\left|\xi\right|\right)
\ \ \ \mbox{for all}\ \ \xi\in\mathbb{R}.
\end{equation}
Then there exist functions $\sigma$ and $\nu$  in $\Sr$ 
%% and an infinitely differentiable function
%% $\nu$ 
such that $\sigma$ is a solution of the nonlinear integral equation
\begin{equation}
\sigma(x) = S\left[T\left[\sigma\right]\right](x) + p(x) + \nu(x),
\ \ \ \mbox{for all}\ \ x\in\mathbb{R},
\label{main_theorem:inteq}
\end{equation}
\begin{equation}
\left|\widehat{\sigma}(\xi)\right| \leq
\left(1+\frac{2\Gamma}{\lambda}\right) \Gamma
\exp\left(- \mu|\xi|\right) 
\ \ \ \mbox{for all}\ \ \left|\xi\right| \leq \sqrt{2}\lambda,
\label{main_theorem:sigma1}
\end{equation}
\begin{equation}
\widehat{\sigma}(\xi) = 0 
\ \ \ \mbox{for all}\ \ \left|\xi\right| > \sqrt{2}\lambda,
\label{main_theorem:sigma2}
\end{equation}
and
\begin{equation}
\|\nu\|_\infty \leq 
\frac{\Gamma}{2\mu}
 \left(1+\frac{4\Gamma}{\lambda}\right)
 \exp\left(-\mu\lambda\right).
\label{main_theorem:nu}
\end{equation}
\label{main_theorem}
\end{theorem}
%%%%%%%%%%%%%%%%%%%%%%%%%%%%%%%%%%%%%%%%%%%%%%%%%%%%%%%%%%%%%%%%%%%%%%%%%%

\begin{remark}
By combining (\ref{main_theorem:lambda}) with (\ref{main_theorem:sigma1}),
we see that
\begin{equation}
\left|\widehat{\sigma}(\xi)\right| 
\leq 
2\Gamma \exp\left(-\mu|\xi|\right) 
\end{equation}
for all $|\xi| \leq \sqrt{2}\lambda$.  Similarly, 
we conclude from (\ref{main_theorem:lambda}) and (\ref{main_theorem:nu})
that
\begin{equation}
\|\nu\|_\infty \leq 
\frac{3\Gamma}{2\mu}
 \exp\left(-\mu\lambda\right).
\end{equation}
\end{remark}

Suppose that $\sigma$ and $\nu$ are the functions obtained by invoking
Theorem~\ref{main_theorem}, and that 
$x(t)$ is the function defined by the formula
\begin{equation}
x(t) = \int_a^t \sqrt{q(u)}\ du.
\end{equation}
%
%% and that $\tilde{q}$ is a solution of the ordinary differential equation
%% %
%% \begin{equation}
%%  \frac{1}{\tilde{q}(t)} \left(
%% \frac{5}{4} \left(\frac{\tilde{q}'(t)}{\tilde{q}(t)}\right)^2
%% -\frac{\tilde{q}''(t)}{\tilde{q}(t)} \right) = p(t) + \nu(t)
%% \ \ \ \mbox{for all}\ \ a \leq t \leq b.
%% \label{overview:qtilde}
%% \end{equation}
%% %
We define  $\delta$ by the formula
\begin{equation}
\delta(x) = T\left[\sigma\right](x),
\end{equation}
$r$ by the formula
\begin{equation}
r(t) = \log(q(t)) + \delta(x(t)),
\end{equation}
and $\alpha$ by the formula
\begin{equation}
\alpha(t) = \lambda \int_a^t  \exp\left(\frac{r(u)}{2}\right)\ du.
\label{overview:alpha}
\end{equation}
From the discussion in Section~\ref{section:inteq}, we conclude that
$\delta(x)$ is a solution of the nonlinear differential equation
\begin{equation}
\delta''(x) + 4\lambda^2 \delta(x) = S\left[\delta\right](x) + p(x) + \nu(x)
\ \ \ \mbox{for all} \ \ x\in\mathbb{R},
\end{equation}
%
%% that $\delta(x(t))$ is a solution of the nonlinear differential equation
%% %
%% \begin{equation}
%% \delta''(t) - \frac{1}{2}\frac{q'(t)}{q(t)}\delta'(t) 
%% -\frac{1}{4}\left(\delta'(t)\right)^2
%% + 4\lambda^2 
%% q(t) \left(\exp(\delta(t))-1\right) 
%% =
%% q(t)\left( p(t) + \nu(t)\right) 
%% \ \ \ \mbox{for all}\ \ \ t\in\mathbb{R},
%% \end{equation}
%
that $r(t)$ is a solution of the nonlinear differential equation
\begin{equation}
r''(t) - \frac{1}{4}(r'(t))^2 + 4\lambda^2  \left(\exp(r(t))-q(t)\right) =  q(t) \nu(t)
\ \ \ \mbox{for all}\  \ t\in\mathbb{R},
\end{equation}
and that $\alpha$ is a solution of the nonlinear differential equation
\begin{equation}
\left(\alpha'(t)\right)^2 = \lambda^2 
\left(\frac{\nu(t)}{4\lambda^2}+1\right)q(t) - \frac{1}{2}\frac{\alpha'''(t)}{\alpha'(t)}
+ \frac{3}{4} \left(\frac{\alpha''(t)}{\alpha'(t)}\right)^2
\ \ \ \mbox{for all}\ \ t \in\mathbb{R}.
\label{overview:kummer}
\end{equation}
From (\ref{overview:kummer}), we see that $\alpha$ is a phase function
for the second order linear  ordinary differential equation
\begin{equation}
y''(t) + \lambda^2 \left(1+\frac{\nu(t)}{4\lambda^2}\right)q(t)y(t) = 0
\ \ \ \mbox{for all}\ \ a \leq t \leq b.
\label{overview:perturbed_equation}
\end{equation}
Since the magnitude of $\nu$ is small, we expect that the difference between solutions of 
(\ref{overview:perturbed_equation}) and those of 
(\ref{introduction:original_equation}) will be small as well.
Indeed, the following theorem follows easily by applying
the Liouville-Green transform and invoking Gronwall's inequality.

\vskip 1em
\begin{theorem}
Suppose that $q$ is continuous and strictly positive on the interval $[a,b]$,
that $f$ is a continuous function $[a,b] \to \mathbb{C}$, 
and that $\lambda$ is a positive real number.
Suppose also that $z:[a,b]\to\mathbb{C}$ is a twice continuously differentiable
function  such that
\begin{equation}
z''(t) + \lambda^2 q(t) z(t) = 0
\ \ \ \mbox{for all}\  \ a\leq t \leq b,
\end{equation}
and that $y:[a,b]\to\mathbb{C}$ is a twice continuously differentiable function 
such that
\begin{equation}
y''(t) + \lambda^2 q(t) y(t) = f(t)
\ \ \ \mbox{for all}\  \ a\leq t \leq b,
\end{equation}
$y(a)  = z(a)$, and $y'(a) = z'(a)$.
Then there exists a positive real number $C$ such that
\begin{equation}
\left|y(t)-z(t)\right| \leq \frac{C}{ \lambda }
\sup_{a\leq t \leq b} \left|f(t)\right|
\ \ \ \mbox{for all}\ \ a\leq t \leq b.
\end{equation}
The constant $C$ depends on $q$ but not on $\lambda$ or $f$.
\label{overview:theorem2}
\end{theorem}
\begin{proof}
We define the function $\Delta$ by the formula
\begin{equation}
\Delta(t) = y(t) - z(t)
\label{overview:theorem1:1}
\end{equation}
and observe that $\Delta$ is the unique solution of the boundary value problem
\begin{equation}
\left\{
\begin{aligned}
\Delta''(t) + \lambda^2 q(t) \Delta(t)  &= f(t)
\ \ \ \mbox{for all}\ \ a \leq t \leq b\\
\Delta(a) = \Delta'(a) &= 0.
\end{aligned}
\right.
\label{overview:theorem1:2}
\end{equation}
We define the function  $x(t)$ via the formula
\begin{equation}
x(t) = \int_a^t \sqrt{q(u)}\ du
\label{overview:theorem1:4}
\end{equation}
and use $t(x)$ to denote its inverse; moreover, we define 
the function  $p(t)$ via the formula
\begin{equation}
p(t) = \frac{1}{q(t)} \left(\frac{5}{4}\left(\frac{q'(t)}{q(t)}\right)^2 - \frac{q''(t)}{q(t)}
\right).
\label{overview:theorem1:6}
\end{equation} 
According to Theorem~\ref{preliminaries:liouville-green:theorem1}, the function
$\varphi$ defined via
\begin{equation}
\varphi(x)  = \left(q(t(x))\right)^{1/4} \Delta(t(x)),
\label{overview:theorem1:3}
\end{equation}
is the unique solution of the initial value problem
\begin{equation}
\left\{
\begin{aligned}
\varphi''(x) + \lambda^2 \varphi(x)  &= 
 \left(q(x)\right)^{-3/4} f(x)
- \frac{1}{4} p(x) \varphi(x)
\ \ \ \mbox{for all}\ \ 0 \leq t \leq x(b)\\
\varphi(0) = \varphi'(0) &= 0.
\end{aligned}
\right.
\label{overview:theorem1:5}
\end{equation}
We conclude from (\ref{overview:theorem1:5})
and Theorem~\ref{preliminaries:helmholtz:theorem3} that
\begin{equation}
\varphi(x) = 
\frac{1}{\lambda} \int_{0}^x \sin\left(\lambda(x-u)\right) 
\left(
 \left(q(u)\right)^{-3/4} f(u)
- \frac{1}{4} p(u) \varphi(u)
\right)\ du
\label{overview:theorem1:7}
\end{equation}
for all $0 \leq x \leq x(b)$.  We let
\begin{equation}
C_1 = x(b)\cdot \sup_{a \leq u \leq b} (q(u))^{-3/4}
\label{overview:theorem1:8}
\end{equation}
and 
\begin{equation}
\|f\|_\infty = \sup_{a \leq u \leq b} \left|f(u)\right|.
\label{overview:theorem1:9}
\end{equation}
We observe that 
\begin{equation}
\left|
\ \frac{1}{\lambda} \int_{0}^x \sin\left(\lambda(x-u)\right) 
\left(q(u)\right)^{-3/4} f(u)\ du
\ \right|
\leq 
\frac{C_1}{\lambda}\ \|f\|_\infty
\label{overview:theorem1:10}
\end{equation}
for all $0 \leq x \leq x(b)$.  Now we let
\begin{equation}
C_2 = \sup_{a \leq u \leq b} p(u)
\label{overview:theorem1:11}
\end{equation}
and observe that
\begin{equation}
\left|
\ \frac{1}{4 \lambda} \int_{0}^x \sin\left(\lambda(x-u)\right) 
 p(u) \phi(u)\ du
\ \right|
\leq 
\ \frac{C_2}{4 \lambda} \int_{0}^x 
 \left| \phi(u)\right|\ du
\label{overview:theorem1:12}
\end{equation}
for all $0 \leq x \leq x(b)$.  We combine (\ref{overview:theorem1:7}),
(\ref{overview:theorem1:10}) and (\ref{overview:theorem1:12})
in order to conclude that 
\begin{equation}
\left|\varphi(x)\right| \leq 
\frac{C_1}{\lambda} \|f\|_\infty
+
\ \frac{C_2}{4 \lambda} \int_{0}^x 
 \left| \phi(u)\right|\ du
\label{overview:theorem1:13}
\end{equation}
for all $0 \leq x \leq x(b)$.  
By invoking Gronwall's inequality (which is Theorem~\ref{preliminaries:gronwall:theorem1}
in Section~\ref{section:preliminaries:gronwall})
we conclude that
\begin{equation}
\left|\varphi(x)\right| \leq 
\frac{C_1}{\lambda} \|f\|_\infty \exp\left(
\frac{C_2 x}{4 \lambda}
\right)
\label{overview:theorem1:14}
\end{equation}
for all $0 \leq x \leq x(b)$.  We conclude from (\ref{overview:theorem1:1}),
(\ref{overview:theorem1:3}) and  (\ref{overview:theorem1:14}) that
\begin{equation}
\left|y(t) - z(t)\right| \leq 
\frac{C}{\lambda} \|f\|_\infty
\end{equation}
for all $a \leq t \leq b$, where $C$ is defined via the formula
\begin{equation}
C = C_1 \exp\left(\frac{C_2 x(b)}{4 \lambda} \right)\cdot
\sup_{a \leq u \leq b} \left(q(u)\right)^{-1/4}.
\end{equation}
\end{proof}

By applying Theorem~\ref{overview:theorem2} to (\ref{overview:perturbed_equation})
and (\ref{introduction:original_equation})  we obtain the following.
\vskip 1em
\begin{theorem}
Suppose that the hypotheses of Theorem~\ref{main_theorem} are satisfied,
that $\sigma$ and $\nu$ are the functions obtained by invoking
it.  Suppose also that $\alpha$ is defined as in (\ref{overview:alpha}),
and that $u$, $v$ are the functions defined via the formulas
\begin{equation}
u(t) = \frac{\cos(\alpha(t))}{\sqrt{\alpha'(t)}}
\label{overview:u}
\end{equation}
and
\begin{equation}
v(t) = \frac{\sin( \alpha(t))}{\sqrt{\alpha'(t)}}.
\label{overview:v}
\end{equation}
Then there exist a constant $C$ and a basis $\{\tilde{u},\tilde{v}\}$ in the space
of solutions of (\ref{introduction:original_equation}) such that 
\begin{equation}
\left|u(t) - \tilde{u}(t) \right| \leq \frac{C}{\lambda} \exp\left(-\mu\lambda\right)
\ \ \ \mbox{for all}\  \ a \leq t \leq b
\end{equation}
and
\begin{equation}
\left|v(t) - \tilde{v}(t) \right| \leq \frac{C}{\lambda} \exp\left(-\mu\lambda\right)
\ \ \ \mbox{for all}\  \ a \leq t \leq b.
\end{equation}
The constant $C$ depends on the coefficient $q$ appearing in (\ref{introduction:original_equation}),
but not on the parameter $\lambda$.
\end{theorem}

The rest of this article is devoted to the 
proof of Theorem~\ref{main_theorem}.  It is divided among Sections~\ref{section:bandlimit},
 \ref{section:spectrum} and \ref{section:solution}.   
The principal difficulty lies in  constructing
a function $\nu$ such that (\ref{overview:inteq}) admits a solution.
We accomplish this by introducing a modified integral equation
\begin{equation}
\sigma_b(x) = \OpS{\OpTb{\sigma_b}}(x) + p(x),
\label{overview:modinteq}
\end{equation}
where $T_b$ is a ``band-limited'' version of $T$.  That is, $\OpTb{f}$ is defined
via the formula
\begin{equation}
\widehat{\OpTb{f}}(\xi) = \widehat{\OpT{f}}(\xi) \widehat{b}(\xi),
\end{equation}
where $\widehat{b}(\xi)$ is the $C_c^{\infty}\left(\mathbb{R}\right)$ bump function
given by Formula~(\ref{preliminaries:bump:b1}) of Section~\ref{section:preliminaries:bump}.
In Section~\ref{section:bandlimit},
%% we introduce we show that under
%% mild conditions on $p$ and $\lambda$, the nonlinear integral equation
%% (\ref{overview:inteq_modified}) admits a solution $\sigma_b$.
%% The argument proceeds by 
we apply the Fourier transform to (\ref{overview:modinteq})
and use the contraction mapping principle to show that under mild
conditions on $p$ and $\lambda$ the resulting equation admits a  solution.
This gives rise to a solution $\sigma_b$ of (\ref{overview:modinteq}).
In Section~\ref{section:spectrum},  we show that 
if the Fourier transform of the function $p$ is exponentially decaying, then
the Fourier transform of  $\sigma_b$  is as well.
In Section~\ref{section:solution} we use the solution $\sigma_b$ of (\ref{overview:modinteq})
in order to construct functions $\sigma$ and $\nu$ which satisfy
(\ref{overview:inteq}).  Moreover, we show that $\sigma$ can be
taken to be an element of the space $\Sr$ of rapidly decaying Schwartz functions
(see Section~\ref{section:preliminaries:schwartz}), that the Fourier transform
of $\sigma$ is compactly supported and exponentially decaying,
and that $\nu$ is an element of $\Sr$
whose $\Lp{\infty}$ norm  decays exponentially with $\lambda$.

\label{section:overview}
\end{section}

%% file: bandlimit.tex
\begin{section}{Band-limited integral equation}

In this section, we introduce a ``band-limited'' 
version of the operator $T$ , use it to form an alternative
to the integral equation (\ref{inteq:integral_equation}),
and apply the contraction mapping principle in order
to show that this alternate equation admits a solution under
mild conditions on the function $p$ and the parameter $\lambda$.

Let  $\widehat{b}(\xi)$  be the bump function defined
via formula (\ref{preliminaries:bump:b1}) so that
\begin{enumerate}
\item
$
\begin{aligned}
\widehat{b}(\xi) = 1 \ \ &\mbox{for all}\ \ |\xi| \leq \lambda,
\end{aligned}
$

\item
$
\begin{aligned}
0 \leq \widehat{b}(\xi) \leq 1  \ \ &\mbox{for all}\ \ \xi\in \mathbb{R},
\ \ \mbox{and} 
\end{aligned}
$

\item
$\widehat{b}$ is supported on a proper subset of $(-\sqrt{2}\lambda,\sqrt{2}\lambda)$.
%% there exists an $\delta >0$ such that
%% $
%% \begin{aligned}
%% \widehat{b}(\xi) = 0 \ \ &\mbox{for all}\ \ |\xi| > \sqrt{2}\lambda-\delta.
%% \end{aligned}
%% $
\label{item3}
\end{enumerate}
We define the operator $\OpTb{f}$ for functions
$f \in C_0\left(\mathbb{R}\right)$ such that $\widehat{f} \in \Lp{1}$ via the formula
\begin{equation}
\widehat{\OpTb{f}}(\xi) 
= 
\widehat{f}(\xi)\frac{\widehat{b}(\xi)}{4\lambda^2-\xi^2}.
\label{bandlimit:definition_of_Tb}
\end{equation}
We will refer to $T_b$ as the band-limited version of the operator $T$ and 
and we call the nonlinear integral equation
\begin{equation}
\sigma_b(x) = 
\OpS{\OpTb{\sigma_b}}(x) + p(x)
\ \ \ \mbox{for all}\ \ x\in\mathbb{R}
\label{bandlimit:integral_equation}
\end{equation}
obtained by replacing $T$ with $T_b$ in (\ref{inteq:integral_equation})  
the ``band-limited'' version of  (\ref{inteq:integral_equation}).  
\vskip  1em
\begin{remark}
The function
\begin{equation}
\frac{\widehat{b}(\xi)}{4\lambda^2-\xi^2}
\label{bandlimit:definition_of_Tb2}
\end{equation}
is an element of $C_c^\infty\left(\mathbb{R}\right)$ since the support of $\widehat{b}$
is bounded away from the points $\pm 2\lambda$ at which the denominator vanishes.
Consequently, (\ref{bandlimit:definition_of_Tb}) is a compactly supported
tempered distribution and
$\OpTb{f}$ is an entire function whenever $f$ is a tempered distribution.
For our purposes, it suffices to know that $\OpTb{f}$ is defined
for $\widehat{f} \in \Lp{1}$.
\end{remark}

It is convenient to analyze 
(\ref{bandlimit:integral_equation}) in the Fourier domain rather than the space domain.
We denote by  $W_b$ and $\widetilde{W}_b$ the linear operators defined 
for $f \in \Lp{1}$  via the formulas
\begin{equation}
\OpWb{f}(\xi) = 
f(\xi)
\frac{\widehat{b}(\xi) }{4\lambda^2-\xi^2}
\label{bandlimit:definition_of_W}
\end{equation}
and
\begin{equation}
\OpWTb{f}(\xi) = 
f(\xi)
\frac{- i\xi \widehat{b}(\xi)}{4\lambda^2-\xi^2},
\label{bandlimit:definition_of_W'}
\end{equation}
where $\widehat{b}(\xi)$ is the function used to define the operator
$T_b$.  
 We  define functions $\psi(\xi)$ and $w(\xi)$ using the formulas
\begin{equation}
\psi(\xi) = \widehat{\sigma_b}(\xi)
\end{equation}
and
\begin{equation}
w(\xi) = \widehat{p}(\xi).
\label{bandlimit:def_of_v}
\end{equation}
Finally,  
we define $R\left[f\right]$ for functions $f \in \Lp{1}$ via
\begin{equation}
R\left[f\right](\xi) = 
\frac{1}{8\pi}\OpWTb{f}*\OpWTb{f}(\xi)
- 4 \lambda^2 
\exp_2^*\left[\OpWb{f}\hskip .2em\right](\xi)
+ w(\xi),
\label{bandlimit:definition_of_R}
\end{equation}
where $\exp_2^*$ is the operator defined by Formula~(\ref{preliminaries:convolution:exp2})
of Section~\ref{section:preliminaries:convolution}.
Applying Fourier transform to both sides of (\ref{bandlimit:integral_equation})
results in the nonlinear equation
\begin{equation}
\psi(\xi) =
R\left[\psi\right](\xi).
\label{bandlimit:equation}
\end{equation}

%%%%%%%%%%%%%%%%%%%%%%%%%%%%%%%%%%%%%%%%%%%%%%%%%%%%%%%%%%%%%%

The following theorem gives conditions
under which the sequence $\{\psi_n\}_{n=0}^\infty$ of fixed point iterates 
for (\ref{bandlimit:equation}) obtained by using the function $w$ 
defined by (\ref{bandlimit:def_of_v}) as an initial approximation  converges.
More explicitly,  $\psi_0$ is defined by the formula
\begin{equation}
\psi_0(\xi) = w(\xi),
\label{convergence:psi0}
\end{equation}
and for each integer $n \geq 0$, $\psi_{n+1}$  is obtained from $\psi_n$ via 
\begin{equation}
\psi_{n+1}(\xi) = 
R\left[\psi_{n}\right](\xi).
\label{convergence:psin}
\end{equation}

%%%%%%%%%%%%%%%%%%%%%%%%%%%%%%%%%%%%%%%%%%%%%%
%
%
%%%%%%%%%%%%%%%%%%%%%%%%%%%%%%%%%%%%%%%%%%%%%

\vskip 1em
\begin{theorem}
Suppose that $\lambda > 0$ is a real number,  and that $w$ is an element 
of $\Lp{1}$ such that
\begin{equation}
\|w\|_1 \leq \frac{\pi}{2} \lambda^2.
\label{convergence:theorem1:assumption}
\end{equation}
Then the sequence  $\{\psi_n\}$  defined by (\ref{convergence:psi0})
and (\ref{convergence:psin}) converges in $\Lp{1}$ norm to a
function $\psi \in \Lp{1}$ such that
\begin{equation}
\psi(\xi) = R\left[\psi\right](\xi)
\ \ \ \mbox{for all} \ \ \xi\in\mathbb{R}.
\label{convergence:theorem1:Req}
\end{equation}
%
%% If, in addition to satisfying (\ref{convergence:theorem1:assumption}),
%% $w$ is an element of $\Lp{\infty}$ such that
%% %
%% \begin{equation}
%% \|w\|_\infty \leq \frac{\pi}{2} \lambda^2,
%% \end{equation}
%% %
%% then the sequence $\left\{\psi_n\right\}$ also converges to $\psi$ in 
%% $\Lp{\infty}$.
\label{convergence:theorem1}
\end{theorem}
%
%
%%%%%%%%%%%%%%%%%%%%%%%%%%%%%%%%%%%%%%%%%%%%%%%55
\begin{proof}
We observe that the Fr\'echet derivative 
(see Section~\ref{section:preliminaries:frechet})
of $R$ at  $f$ is the linear operator  $R_f':\Lp{1}\to\Lp{1}$ given by the formula
\begin{equation}
R'_f\left[h\right](\xi) = 
\frac{\OpWTb{f}*\OpWTb{h}(\xi)}{4\pi}
- 
4\lambda^2
\exp_1^*\left[\frac{\OpWb{f}}{2\pi}\right]*\OpWb{h}
(\xi).
\label{convergence:definition_of_R'}
\end{equation}
From formulas~(\ref{bandlimit:definition_of_W}) and (\ref{bandlimit:definition_of_W'})
and the definition of $\widehat{b}(\xi)$ we see that
\begin{equation}
\left\|\OpWb{f}\right\|_1 \leq \frac{\|f\|_1}{2\lambda^2},
 \label{convergence:Sbounds1}
\end{equation}
and
\begin{equation}
\left\|\OpWTb{f}\right\|_1 \leq \frac{\|f\|_1}{\sqrt{2}\lambda}
\label{convergence:Sbounds2}
\end{equation}
for all $f \in \Lp{1}$.  
We combine (\ref{convergence:definition_of_R'})  with
(\ref{preliminaries:convolution:ac1}) in order to conclude that
\begin{equation}
\begin{aligned}
\left\|R'_f\left[h\right]\right\|_1
&\leq 
\frac{1}{4\pi}
\left\|\OpWTb{f}\right\|_1
\left\|\OpWTb{h}\right\|_1
+ \frac{2 \lambda^2}{\pi}
\left\|\OpWb{f}\right\|_1
\exp\left(\frac{\left\|\OpWb{f}\right\|_1}{2\pi}\right)
\left\|\OpWb{h}\right\|_1
\end{aligned}
\label{convergence:Rprime1}
\end{equation}
for all $f$ and $h$ in $\Lp{1}$.
By inserting (\ref{convergence:Sbounds1}) and (\ref{convergence:Sbounds2})
into (\ref{convergence:Rprime1}) we see that
\begin{equation}
\begin{aligned}
\left\|R'_f\left[h\right]\right\|_1
&\leq 
\frac{\|f\|_1 \|h\|_1}{8\pi\lambda^2}
+ \frac{2\lambda^2}{\pi}
\frac{\|f\|_1}{2\lambda^2}
\frac{\|h\|_1}{2\lambda^2}
\exp\left(\frac{\|f\|_1}{4\pi\lambda^2}\right)
\\
&\leq 
\left(
\frac{\|f\|_1}{8\pi\lambda^2}
+
\frac{\|f\|_1}{2\pi \lambda^2}
\exp\left(\frac{\|f\|_1}{4\pi \lambda^2}\right)
\right)\|h\|_1
\end{aligned}
\label{convergence:bound_R'}
\end{equation}
for all $f$ and $h$ in $\Lp{1}$.
Similarly, by combining (\ref{bandlimit:definition_of_R}), (\ref{preliminaries:convolution:ac2}),
(\ref{convergence:Sbounds1}) and (\ref{convergence:Sbounds2})
we conclude that
\begin{equation}
\begin{aligned}
\left\|\OpR{f}\right\|_1
&\leq 
\frac{1}{8\pi}
\left\|\OpWTb{f}\right\|_1^2
+ \frac{\lambda^2}{\pi}
\left\|\OpWb{f}\right\|_1^2
\exp\left(\frac{\left\|\OpWb{f}\right\|_1}{2\pi}\right)
+ \|w\|_1
\\
&\leq 
\frac{\|f\|_1^2}{16\pi\lambda^2}
+ 
 \frac{\|f\|_1^2}{4\pi \lambda^2}
\exp\left(\frac{\|f\|_1}{4\pi\lambda^2}\right)
+ \|w\|_1
\end{aligned}
\label{convergence:bound_R}
\end{equation}
whenever $f \in \Lp{1}$.  We now let  $r = \pi \lambda^2$
and denote by $\Omega$ the  closed ball of radius $r$ centered at $0$ in $\Lp{1}$.
Suppose that $f \in \Lp{1}$ such that
\begin{equation}
\|f\|_1 \leq r  = \pi \lambda^2,
\label{convergence:r}
\end{equation}
and that 
\begin{equation}
\|w\|_1 \leq \frac{r}{2} = \frac{\pi \lambda^2}{2}.
\label{convergence:v}
\end{equation}
We insert  (\ref{convergence:r}) and (\ref{convergence:v}) into (\ref{convergence:bound_R})
in order to obtain
\begin{equation}
\begin{aligned}
\left\|\OpR{f}\right\|_1
&\leq 
\frac{r^2}{16\pi\lambda^2}
+ 
 \frac{r^2}{4\pi \lambda^2}
\exp\left(\frac{\|f\|_1}{4\pi\lambda^2}\right)
+ \frac{r}{2}
\\
&=
\left(
\frac{r}{16\pi\lambda^2}
+ 
 \frac{r}{4\pi \lambda^2}
\exp\left(\frac{\|f\|_1}{4\pi\lambda^2}\right)
+ \frac{1}{2}\right)r 
\\
&=
\left(
\frac{1}{16}
+ 
 \frac{1}{4}
\exp\left(\frac{1}{4}\right)
+ \frac{1}{2}\right)r 
\\
&\leq
\frac{9}{10}r,
\end{aligned}
\end{equation}
from which we conclude that $R$ maps $\Omega$ into itself.  Next, we insert
(\ref{convergence:r}) into (\ref{convergence:bound_R'})
in order to obtain
\begin{equation}
\begin{aligned}
\left\|R'_f\left[h\right]\right\|_1
&\leq 
\left(
\frac{r}{8\pi\lambda^2}
+ \frac{r}{2\pi \lambda^2} \exp\left(\frac{r}{4\pi\lambda^2}\right)
\right) \|h\|_1
\\
&\leq 
\left(
\frac{1}{8}
+
\frac{1}{2}
\exp\left(\frac{1}{4}\right)
\right)\|h\|_1
\\
&\leq 
\frac{8}{10} \|h\|_1,
\end{aligned}
\end{equation}
which shows that $R$ is a contraction on $\Omega$. 
We now invoke the contraction mapping theorem 
(Theorem~\ref{preliminaries:frechet:cmp}
in Section~\ref{section:preliminaries:frechet})
in order to conclude that
any sequence of fixed point iterates  for (\ref{bandlimit:equation})
which originates  in $\Omega$  
will converge in $\Lp{1}$ to a solution of (\ref{bandlimit:equation}).  
Since $\{\psi_n\}$ is such a sequence, it converges in $\Lp{1}$
to a function $\psi$ such that 
\begin{equation}
\psi(\xi) = R\left[\psi\right](\xi)
\ \ \ \mbox{for almost all}\ \ \xi\in\mathbb{\mathbb{R}}.
\label{convergence:111}
\end{equation}
We change the values of $\psi$ on a set of measure zero (which does
not affect $R\left[\psi\right]$) in order to ensure that
(\ref{convergence:111}) holds for all $\xi\in\mathbb{R}$.
\end{proof}

\label{section:bandlimit}
\end{section}

%% file: spectrum.tex
\begin{section}{Fourier estimate}

In this section, we derive a pointwise estimate on the solution
$\psi$ of Equation~(\ref{bandlimit:equation}) 
under additional assumptions on the function $w$.

%%%%%%%%%%%%%%%%%%%%%%%%%%%%%%%%%%%%%%%%%%%%%%%%%%%%%%%%%%%%%%%%%%%%%%%%%%%%%%%%%%%%%%%%%%%%%%
%  S Lemma 2
%
%%%%%%%%%%%%%%%%%%%%%%%%%%%%%%%%%%%%%%%%%%%%%%%%%%%%%%%%%%%%%%%%%%%%%%%%%%%%%%%%%%%%%%%%%%%%%%
\vskip 1em
\begin{lemma}
Suppose that $\mu$ and $C$ are real numbers such that
\begin{equation}
0 \leq C < \mu.
\label{spectrum:S2:assumption1}
\end{equation}
Suppose also that  $f \in \Lp{1}$, and that
\begin{equation}
\left|f(\xi)\right| \leq C \exp(-\mu |\xi|)
\ \ \ \mbox{for all}\ \ \xi\in\mathbb{R}.
\label{spectrum:S2:assumption2}
\end{equation}
Then
\begin{equation}
\left|\exp_2^*\left[f\right](\xi)\right|
\leq 
\frac{C^2}{2\pi} \exp(-\mu|\xi|)
\frac{1+\mu|\xi|}{\mu}
\exp\left(\frac{C}{2\pi \mu}\right)
\exp
\left(
\frac{C}{2 \pi }|\xi|
\right)
\ \ \ \mbox{for all}\ \ \ \xi\in \mathbb{R},
\label{spectrum:S2:bound}
\end{equation}
where $\exp_2^*$ is the operator defined in (\ref{preliminaries:convolution:exp2}).
\label{spectrum:S2}
\end{lemma}
%%%%%%%%%%%%%%%%%%%%%%%%%%%%%%%%%%%%%%%%%%%%%%%%%%%%%%%%%%%%%%%%%%%%%%%%%%%
\begin{proof}
We let 
%
%% \begin{equation}
%% g(\xi) = C \exp(-\mu|\xi|)
%% \end{equation}
%% %
%% and for each integer $m > 0$, denote by $g_m$ the
%% $m$-fold convolution product of the function $g$.
%% That is to say that $g_1$ is defined via the formula
%% %
\begin{equation}
g_1(\xi) = C \exp(-\mu|\xi|)
\end{equation}
and for each integer $m > 0$,
we define $g_{m+1}$ in terms of $g_m$ through the formula
\begin{equation}
g_{m+1}(\xi) = \frac{1}{2\pi}\ g_{m}*g_1(\xi).
\end{equation}
We observe that for each integer $m > 0$ and all $\xi \in \mathbb{R}$,
\begin{equation}
g_m(\xi)
=
2 \sqrt{\mu C} \left(\frac{C|\xi|}{2\pi}\right)^{m-1/2}
\frac{K_{m-1/2}(\mu|\xi|)}{\Gamma(m)},
\label{spectrum:S2:convolution_product}
\end{equation}
where $K_\nu$ denotes the modified Bessel function of the second kind of order $\nu$
(see Section~\ref{section:preliminaries:bessel}).
By repeatedly applying Theorem~\ref{preliminaries:bessel:theorem1} 
of Section~\ref{section:preliminaries:bessel}, we conclude that
for all integers $m >0$ and all real $t$,
\begin{equation}
\begin{aligned}
K_{m-1/2}(t) 
&\leq K_{1/2}(t) 
\prod_{j=1}^{m-1} \left(\frac{2\left(j-\frac{1}{2}\right)}{t}+1 \right)\\
&=
K_{1/2}(t)
\left(\frac{2}{t}\right)^{m-1}
\frac{\Gamma\left(\frac{1+t}{2}+m-1\right)}{\Gamma\left(\frac{1+t}{2}\right)}.
\end{aligned}
\label{spectrum:S2:kbound1}
\end{equation}
We insert the identity
\begin{equation}
K_{1/2}(t) = \sqrt{\frac{\pi}{2t}} \exp(-t)
\end{equation}
into (\ref{spectrum:S2:kbound1}) in order to conclude that
for all integers $m > 0$ and all real numbers $t > 0$,
\begin{equation}
\begin{aligned}
K_{m-1/2}(t) 
&\leq
\frac{\sqrt{\pi}}{2}  \left(\frac{t}{2}\right)^{1/2-m} \exp(-t) 
\frac{
\Gamma\left( \frac{1+t}{2}  + m-1\right)}
{\Gamma\left( \frac{1+t}{2}\right)}.
\end{aligned}
\label{spectrum:S2:kbound}
\end{equation}
By combining (\ref{spectrum:S2:kbound}) and (\ref{spectrum:S2:convolution_product})  we conclude
that
\begin{equation}
g_m(\xi)
\leq 
C \exp(-\mu|\xi|)\left(\frac{C}{\pi \mu}\right)^{m-1} 
\frac{\Gamma\left(\frac{1+\mu|\xi|}{2}+m-1\right)}
{\Gamma(m)\Gamma\left(\frac{1+\mu|\xi|}{2}\right)}
\label{spectrum:S2:gmbound}
\end{equation}
for all integers $m > 0$ and all $\xi \neq 0$.  
Moreover, the limit
as $\xi \to 0$ of each side of (\ref{spectrum:S2:gmbound}) is finite
and the two limits are equal, so (\ref{spectrum:S2:gmbound}) in fact holds 
for all $\xi \in \mathbb{R}$.  We sum (\ref{spectrum:S2:gmbound}) over $m=2,3,\ldots$ 
in order to conclude that
\begin{equation}
\begin{aligned}
\exp_2^*\left[g\right](\xi)
&\leq 
C \exp(-\mu |\xi|)
\sum_{m=2}^\infty
\left(\frac{C}{\pi \mu}\right)^{m-1}
\frac{\Gamma\left(\frac{1+\mu|\xi|}{2}+m-1 \right)}
{\Gamma(m+1)\Gamma(m)\Gamma\left(\frac{1+\mu|\xi|}{2}\right)}
\\
&=
C \exp(-\mu |\xi|)
\sum_{m=1}^\infty
\left(\frac{C}{\pi \mu}\right)^{m}
\frac{\Gamma\left(\frac{1+\mu|\xi|}{2}+m \right)}
{\Gamma(m+2)\Gamma(m+1)\Gamma\left(\frac{1+\mu|\xi|}{2}\right)}
\end{aligned}
\label{spectrum:S2:interbound}
\end{equation}
for all $\xi \in \mathbb{R}$.     Now we observe that
\begin{equation}
\frac{1}{\Gamma(m+2)} \leq  \left(\frac{1}{2}\right)^{m}
\ \ \ \mbox{for all}\ \ m=0,1,2,\ldots.
\label{spectrum:S2:gamma_bound}
\end{equation}
Inserting (\ref{spectrum:S2:gamma_bound}) into (\ref{spectrum:S2:interbound})
yields
\begin{equation}
\begin{aligned}
\exp_2^*\left[g\right](\xi)
&\leq
C \exp(-\mu |\xi|)
\sum_{m=1}^\infty 
\left(\frac{C}{2 \pi \mu}\right)^{m}
\frac{\Gamma\left(\frac{1+\mu|\xi|}{2}+m\right)}
{\Gamma(m+1)\Gamma\left(\frac{1\mu |\xi|}{2}\right)}
\end{aligned}
\label{spectrum:S2:binomial1}
\end{equation}
for all $\xi \in \mathbb{R}$.    Now we apply the binomial theorem
(Theorem~\ref{preliminaries:binomial_theorem}
 of Section~\ref{section:preliminaries:binomial}),
which is justified since $C < a < 2\pi \mu $, to conclude that
\begin{equation}
\begin{aligned}
\exp_2^*\left[g\right](\xi)
&\leq
C \exp(- \mu|\xi|)
\left(
\left(1-\frac{C}{2\pi \mu}\right)^{-\frac{1+\mu|\xi|}{2}}-1
\right)
\\
&=
C \exp(-\mu|\xi|)
\left(
\exp
\left(
\frac{1+\mu|\xi|}{2}
\log\left(\frac{1}{1-\frac{C}{2\pi \mu}}\right)
\right)-1
\right)
\end{aligned}
\label{spectrum:S2:binomial2}
\end{equation}
for all $\xi \in \mathbb{R}$.  We observe that
\begin{equation}
\exp(x) -1 \leq x \exp(x)
\ \ \ \mbox{for all}\ \ x \geq 0,
\label{spectrum:S2:1}
\end{equation}
and
\begin{equation}
0 \leq \log\left(\frac{1}{1-x}\right) \leq 2x 
\ \ \ \mbox{for all}\ \ 0 \leq x \leq \frac{1}{2\pi}.
\label{spectrum:S2:2}
\end{equation}
By combining (\ref{spectrum:S2:1}) and (\ref{spectrum:S2:2})
with (\ref{spectrum:S2:binomial2}) we conclude that
%% %
\begin{equation}
\begin{aligned}
\exp_2^*\left[g\right](\xi)
&\leq
C \exp(-\mu|\xi|)
\left(\frac{1+\mu|\xi|}{2}\right)
\log\left(\frac{1}{1-\frac{C}{2\pi \mu}}\right)
\exp
\left(
\frac{1+\mu|\xi|}{2}
\log\left(\frac{1}{1-\frac{C}{2\pi \mu}}\right)
\right)
\\
&\leq
\frac{C^2}{2\pi} \exp(-\mu|\xi|)
\left(\frac{1+\mu|\xi|}{\mu}\right)
\exp\left(\frac{C}{2\pi \mu}\right)
\exp
\left(
\frac{C}{2 \pi }|\xi|
\right)
\end{aligned}
\label{spectrum:S2:binomial3}
\end{equation}
%% %
for all $\xi \in\mathbb{R}$.   Note that in (\ref{spectrum:S2:binomial3}),
we used the assumption that $C < a$ in order to apply the inequality
(\ref{spectrum:S2:2}).
Owing to (\ref{spectrum:S2:assumption2}),
\begin{equation}
\left|\exp_2^*\left[f\right](\xi)\right| \leq 
\exp_2^*\left[g\right](\xi)
\ \ \ \mbox{for all} \ \ \xi\in\mathbb{R}.
\end{equation}
By combining this observation with (\ref{spectrum:S2:binomial3}),
we obtain (\ref{spectrum:S2:bound}), which completes the proof.
\end{proof}
%%%%%%%%%%%%%%%%%%%%%%%%%%%%%%%%%%%%%%%%%%%%%%%%%%%%%%%%%%%%%%%%%%%%%%%%%%%%%%%%%
\begin{remark}
Kummer's confluent hypergeometric function $M(a,b,z)$ is defined 
by the series
\begin{equation}
M(a,b,z) = 1 + \frac{az}{b} + 
\frac{(a)_2 z^2}{(b)_2 2!} + 
 \frac{(a)_3 z^3}{(b)_3 3!} + \cdots,
%% M(a,b,z) = 1 + \frac{az}{b} + 
%% \frac{\Gamma(a+k)}{\Gamma{a}}
%% \frac{(a)_2 z^2}{(b)_2 2!} + 
%%  \frac{(a)_3 z^3}{(b)_3 3!} + \cdots,
\label{kummer:confluent}
\end{equation}
where $(a)_n$ is the Pochhammer symbol
\begin{equation}
(a)_n = \frac{\Gamma(a+n)}{\Gamma(a)} = a (a+1) (a+2) \ldots (a+n-1).
\end{equation}
By comparing the definition of $M(a,b,z)$ with (\ref{spectrum:S2:interbound}),
we conclude that
\begin{equation}
\left|\exp_2^*\left[f\right](\xi)\right|
\leq C \exp(-\mu|\xi|)\left(M\left(\frac{1+\mu|\xi|}{2},2,\frac{C}{\pi \mu} \right)-1\right)
\ \ \ \mbox{for all} \ \ \xi\in\mathbb{R}
\label{spectrum:tighter_bound}
\end{equation}
provided
\begin{equation}
\left|f(\xi)\right|\leq C \exp(-\mu|\xi|)
\ \ \ \mbox{for all} \ \ \xi\in\mathbb{R}.
\end{equation}
The weaker bound (\ref{spectrum:S2:bound}) is sufficient for 
our immediate purposes, but formula (\ref{spectrum:tighter_bound}) might serve
as a basis for improved estimates on solutions of Kummer's equation.
\end{remark}

The following lemma is a special case of Formula~(\ref{spectrum:S2:convolution_product}).

%%%%%%%%%%%%%%%%%%%%%%%%%%%%%%%%%%%%%%%%%%%%%%%%%%%%%%%%%%%%%%%%%%%%%%%%%%%%%%%%%%%%%%%%%%%%%%
%  S Lemma 1
%
%%%%%%%%%%%%%%%%%%%%%%%%%%%%%%%%%%%%%%%%%%%%%%%%%%%%%%%%%%%%%%%%%%%%%%%%%%%%%%%%%%%%%%%%%%%%%%
\vskip 1em
\begin{lemma}
Suppose that $C \geq 0$ and $\mu >0 $ are real numbers, and that
$f \in \Lp{1}$ such that
\begin{equation}
\left|f(\xi)\right| \leq C \exp\left(-\mu|\xi|\right)
\ \ \ \mbox{for all}\ \ \xi \in \mathbb{R}.
\label{spectrum:S1:assumption}
\end{equation}
Then
\begin{equation}
\left|f*f(\xi)\right|
\leq
C^2 
\exp(-\mu|\xi|)
\left(\frac{1+\mu|\xi|}{\mu}\right)
\ \ \ \mbox{for all} \ \ \xi \in \mathbb{R}.
\end{equation}
\label{spectrum:S1}
\end{lemma}

%%%%%%%%%%%%%%%%%%%%%%%%%%%%%%%%%%%%%%%%%%%%%%%%%%%%%%%%%%%%%%%%%%%%%%%%%%%%%%%%
We will also make use of the following elementary observation.
\vskip 1em
\begin{lemma}
Suppose that $\mu > 0$ is a real number.  Then
\begin{equation}
\exp(-\mu|\xi|)|\xi| \leq \frac{1}{\mu \exp(1)} \ \ \ \mbox{for all} \ \ \xi\in\mathbb{R}.
\end{equation}
\label{spectrum:exp_lemma}
\end{lemma}
%%%%%%%%%%%%%%%%%%%%%%%%%%%%%%%%%%%%%%%%%%%%%%%%%%%%%%%%%%%%%%%%%%%

We combine Lemmas~\ref{spectrum:S2} and  \ref{spectrum:S1} 
with (\ref{convergence:Sbounds1}) and (\ref{convergence:Sbounds2})
in order to obtain the following key estimate.

%%%%%%%%%%%%%%%%%%%%%%%%%%%%%%%%%%%%%%%%%%%%%%%%%%%%%%%%%%%%%%%%%%%%%%%%%%%%%%%%%%%%%%%%%%%%%%
%
%  S theorem
%
%%%%%%%%%%%%%%%%%%%%%%%%%%%%%%%%%%%%%%%%%%%%%%%%%%%%%%%%%%%%%%%%%%%%%%%%%%%%%%%%%%%%%%%%%%%%%%
\vskip 1em
\begin{theorem}
Suppose that $\Gamma > 0$, $\lambda >0$, $\mu > 0$ and $C > 0 $ are real numbers such that
\begin{equation}
0 \leq C < 2 \mu\lambda^2.
\label{spectrum:STheorem:assumption1}
\end{equation}
Suppose also that  $f \in \Lp{1}$ such that
\begin{equation}
\left|f(\xi)\right| \leq C \exp(-\mu|\xi|)
\ \ \ \mbox{for all}\ \ |\xi| \leq \sqrt{2}\lambda,
\label{spectrum:STheorem:assumption2}
\end{equation}
and that $w \in \Lp{1}$ such that
\begin{equation}
\left|w(\xi)\right| \leq \Gamma \exp(-\mu|\xi|)
\ \ \ \mbox{for all} \ \ \xi\in\mathbb{R}.
\label{spectrum:STheorem:vassumption}
\end{equation}
Suppose further  that $R$ is the operator defined via (\ref{bandlimit:definition_of_R}).
Then
\begin{equation}
\left|\OpR{f}(\xi)\right|
\leq
\exp(-\mu|\xi|)
\left(
\frac{C^2}{\lambda^2}  \left(\frac{1+\mu|\xi|}{\mu}\right)
\left(
\frac{1}{16\pi} + \frac{1}{2\pi}\exp\left(\frac{C}{4\pi \lambda^2 \mu}\right)
\exp\left(\frac{C}{4\pi \lambda^2}|\xi|\right)
\right)
+\Gamma
\right)
\label{spectrum:STheorem:bound}
\end{equation}
for all $\xi \in \mathbb{R}$.
\label{spectrum:STheorem}
\end{theorem}

%%%%%%%%%%%%%%%%%%%%%%%%%%%%%%%%%%%%%%%%%%%%%%%%%%%%%%%%%%%%%%%%%%%%%%%%%%%%%%%%%%%%%%%%
\begin{proof}
We define the operator $R_1$ via the formula
\begin{equation}
R_1\left[f\right](\xi) = \frac{1}{8\pi}
\OpWTb{f}* \OpWTb{f}(\xi)
\end{equation}
and  $R_2$ by the formula
\begin{equation}
R_2\left[f\right](\xi) = -4\lambda^2\exp_2^*\left[\OpWb{f}\right](\xi),
\end{equation}
where $W_b$ and $\widetilde{W}_b$ are defined as in 
Section~\ref{section:bandlimit}.  Then
\begin{equation}
\OpR{f}(\xi) = R_1\left[f\right](\xi) + R_2\left[f\right](\xi) + w(\xi)
\end{equation}
for all $\xi \in \mathbb{R}$.  We observe that
\begin{equation}
\left|\OpWTb{f}(\xi)\right| \leq \frac{C}{\sqrt{2}\lambda} \exp(-\mu|\xi|)
\ \ \ \mbox{for all}\ \ \xi\in\mathbb{R}.
\label{spectrum:STheorem:1}
\end{equation}
By combining Lemma~\ref{spectrum:S1} with (\ref{spectrum:STheorem:1})
we obtain
\begin{equation}
\left|R_1\left[f\right](\xi)\right|
\leq 
\frac{C^2}{16\pi\lambda^2} \exp(-\mu|\xi|)\left(\frac{1+\mu|\xi|}{\mu}\right)
\ \ \ \mbox{for all}\ \ \xi\in\mathbb{R}.
\label{spectrum:STheorem:2}
\end{equation}
Now we observe that
\begin{equation}
\left|\OpWb{f}(\xi)\right| \leq \frac{C}{2\lambda^2} \exp(-\mu|\xi|)
\ \ \ \mbox{for all}\ \ \xi\in\mathbb{R}.
\label{spectrum:STheorem:3}
\end{equation}
Combining Lemma~\ref{spectrum:S2} with (\ref{spectrum:STheorem:3})
yields  
\begin{equation}
\begin{aligned}
\left|R_2\left[f\right](\xi)\right|
&\leq 
\frac{C^2}{2\pi\lambda^2} \exp(-\mu|\xi|)
\left(\frac{1+\mu|\xi|}{\mu}\right)
\exp\left(\frac{C}{4 \pi \lambda^2 \mu}\right)
\exp
\left(
\frac{C}{4\pi \lambda^2 }|\xi|
\right)
\end{aligned}
\label{spectrum:STheorem:4}
\end{equation}
for all $\xi \in \mathbb{R}$.  Note that (\ref{spectrum:STheorem:assumption1})
ensures that the hypothesis (\ref{spectrum:S2:assumption1}) in Lemma~\ref{spectrum:S2} is satisfied.
We combine (\ref{spectrum:STheorem:2})
with (\ref{spectrum:STheorem:4}) and (\ref{spectrum:STheorem:vassumption}) 
in order to obtain (\ref{spectrum:STheorem:bound}), and by so doing
we complete the proof.
\end{proof}
%%%%%%%%%%%%%%%%%%%%%%%%%%%%%%%%%%%%%%%%%%%%%%%%%%%%%%%%%%%%%%%%%%%%%%%%%%%%%%%%%%%%%%%%

\begin{remark}
Note that  Theorem~\ref{spectrum:STheorem} only
requires that $f(\xi)$ satisfy a bound on the interval $[-\sqrt{2}\lambda,\sqrt{2}\lambda]$ 
and not on the entire real line.  
\end{remark}

%%%%%%%%%%%%%%%%%%%%%%%%%%%%%%%%%%%%%%%%%%%%%%%%%%%%%%%%%%%%%%
%
%  THEOREM on the decay of the solution
%
%%%%%%%%%%%%%%%%%%%%%%%%%%%%%%%%%%%%%%%%%%%%%%%%%%%%%%%%%%%%%%

In the following theorem,
we use Theorem~\ref{spectrum:STheorem} to  bound the solution of (\ref{bandlimit:equation}) 
under an assumption on the decay of $w$.

\vskip 1em
\begin{theorem}
Suppose that $\lambda > 0$, $\mu > 0$ 
and $\Gamma > 0$ are real numbers such that
\begin{equation}
\lambda > 2 \max\left\{\Gamma,\frac{1}{\mu}\right\}.
\label{spectrum:theorem2:assumption1}
\end{equation}
Suppose also that $w \in \Lp{1}$ such that
\begin{equation}
\left|w(\xi)\right| \leq  \Gamma \exp(-\mu|\xi|)
\ \ \ \mbox{for all} \ \ \xi\in\mathbb{R}.
\label{spectrum:theorem2:assumption2}
\end{equation}
Then there exists a solution $\psi$ of  (\ref{bandlimit:equation})
in $\Lp{1}$  such that
\begin{equation}
\left|\psi(\xi)\right| \leq
\left(1+\frac{2\Gamma}{\lambda}\right) \Gamma
\exp\left(- \mu|\xi|\right) 
\ \ \ \mbox{for all} \ \ \left|\xi\right| \leq \sqrt{2}\lambda
\label{spectrum:theorem2:bound0}
\end{equation}
and
\begin{equation}
\left|\psi(\xi)\right| \leq
\left(
1+ \frac{4\Gamma}{\lambda} 
\right)\Gamma
\exp\left(-\left(\mu-\frac{1}{2\pi\lambda}\right)\left|\xi\right|\right)
\ \ \ \mbox{for all} \ \ \xi\in\mathbb{R}.
\label{spectrum:theorem2:bound}
\end{equation}
\label{spectrum:theorem2}
\end{theorem}
%%%%%%%%%%%%%%%%%%%%%%%%%%%%%%%%%%%%%%%%%%%%%%%%%%%%%%%%%%%%%%

\begin{proof}
From (\ref{spectrum:theorem2:assumption1}) and
(\ref{spectrum:theorem2:assumption2}) we obtain
\begin{equation}
\|w\|_1 \leq \Gamma \int_{-\infty}^\infty \exp(-\mu\left|\xi\right|)\ d\xi
= \frac{\Gamma}{\mu} < \frac{\lambda^2}{4}.
\label{spectrum:theorem2:v1}
\end{equation}
It follows from Theorem~\ref{convergence:theorem1} and (\ref{spectrum:theorem2:v1})
that  a solution  $\psi(\xi)$ of (\ref{bandlimit:equation}) is obtained
as the limit of the sequence of fixed point iterates $\{\psi_n(\xi)\}$
defined by the formula
\begin{equation}
\psi_0(\xi) = w(\xi)
\end{equation}
and the recurrence
\begin{equation}
\psi_{n+1}(\xi) = \OpR{\psi_{n}}(\xi).
\end{equation}
We now  derive pointwise estimates on the iterates $\psi_n(\xi)$ 
in order to establish (\ref{spectrum:theorem2:bound0}) and 
 (\ref{spectrum:theorem2:bound}).

We denote by $\{\beta_k\}$ be  the sequence of real numbers  generated
by the recurrence relation
\begin{equation}
\beta_{k+1} = \frac{\beta_k^2}{2 \lambda}+\Gamma
\label{spectrum:theorem2:beta1}
\end{equation}
with the initial value
\begin{equation}
\beta_0 = \Gamma.
\end{equation}
From mathematical induction and (\ref{spectrum:theorem2:assumption1}),
we conclude that $\{\beta_k\}$ is a monotonically  increasing sequence which converges
to 
\begin{equation}
\beta = \lambda - \sqrt{\lambda^2-2\Gamma\lambda}.
\end{equation}
We also observe that
\begin{equation}
\beta_n \leq \beta < \left(1+\frac{2\Gamma}{\lambda}\right)\Gamma \leq 2\Gamma.
\label{spectrum:theorem2:beta2}
\end{equation}

Now suppose that $n \geq 0$ is an integer, and that
\begin{equation}
\left|\psi_n(\xi)\right|
\leq \beta_n \exp(-\mu|\xi|)
\ \ \ \mbox{for all}\ \ \ |\xi|\leq\sqrt{2}\lambda.
\label{spectrum:theorem2:bound2}
\end{equation}
When $n=0$, this is simply the assumption (\ref{spectrum:theorem2:assumption2}).
The function $\psi_{n+1}(\xi)$ is obtained from $\psi_n(\xi)$ via the formula
\begin{equation}
\psi_{n+1}(\xi) = \OpR{\psi}(\xi).
\label{spectrum:theorem2:psi}
\end{equation}
We combine Theorem~\ref{spectrum:STheorem} with (\ref{spectrum:theorem2:psi})
and (\ref{spectrum:theorem2:bound2}) to conclude that
\begin{equation}
\begin{aligned}
\left|\psi_{n+1}(\xi)\right|
\leq
\exp(-\mu|\xi|)
\left(
\frac{\beta_n^2}{\lambda^2}  \left(\frac{1+\mu|\xi|}{\mu}\right)
\left(
\frac{1}{16\pi} + \frac{1}{2\pi}\exp\left(\frac{\beta_n}{4\pi \lambda^2 \mu}\right)
\exp\left(\frac{\beta_n}{4\pi \lambda^2}|\xi|\right)
\right)
+\Gamma
\right)
\end{aligned}
\label{spectrum:theorem2:2}
\end{equation}
for all $\xi \in \mathbb{R}$.     The hypothesis (\ref{spectrum:STheorem:assumption1}) 
of  Theorem~\ref{spectrum:STheorem} is satisfied since
\begin{equation}
\beta_n \leq 2\Gamma \leq 2 \lambda^2 \mu
\end{equation}
%% %
for all integers $n \geq 0$.   
We restrict  $\xi$ to the interval $[-\sqrt{2}\lambda,\sqrt{2}\lambda]$ in 
(\ref{spectrum:theorem2:2}) and use the fact that
\begin{equation}
\frac{1}{\mu \lambda} < \frac{1}{2},
\end{equation}
which is a consequence of (\ref{spectrum:theorem2:assumption1}), in order to conclude that
\begin{equation}
\begin{aligned}
\left|\psi_{n+1}(\xi)\right|
&\leq
\exp(-\mu|\xi|)
\left(
\frac{\beta_n^2}{\lambda^2}  \left(\frac{1+\mu\sqrt{2}\lambda}{\mu}\right)
\left(
\frac{1}{16\pi} + \frac{1}{2\pi}\exp\left(\frac{\beta_n}{4\pi \lambda^2 \mu}\right)
\exp\left(\frac{\beta_n}{4\pi \lambda^2}\sqrt{2}\lambda\right)
\right)
+\Gamma
\right)
\\
&\leq
\exp(-\mu|\xi|)
\left(
\frac{\beta_n^2}{\lambda}  
\left(\frac{1}{2} + \sqrt{2}\right)
\left(
\frac{1}{16\pi} + \frac{1}{2\pi}\exp\left(\frac{\beta_n}{8\pi \lambda}\right)
\exp\left(\frac{\beta_n}{2\sqrt{2}\pi \lambda}\right)
\right)
+\Gamma
\right)
\end{aligned}
\label{spectrum:theorem2:3}
\end{equation}
for all $|\xi|\leq \sqrt{2}\lambda$.    
By combining (\ref{spectrum:theorem2:3}) with inequality
\begin{equation}
\frac{\beta_n}{\lambda} \leq \frac{2\Gamma}{\lambda} \leq 1
\ \ \ \mbox{for all}\ \ n\geq 0
\end{equation}
and the observation that 
\begin{equation}
\left(\frac{1}{2} + \sqrt{2}\right)
\left(
\frac{1}{16\pi} + \frac{1}{2\pi}\exp\left(\frac{1}{8\pi}\right)
\exp\left(\frac{1}{2\sqrt{2}\pi}\right)
\right)
< \frac{1}{2},
\label{spectrum:theorem2:5}
\end{equation}
we arrive at the inequality
\begin{equation}
\begin{aligned}
\left|\psi_{n+1}(\xi)\right|
&\leq
\left(\frac{\beta_n^2}{2 \lambda}+\Gamma \right)
\exp(-\mu|\xi|)
\\
&=
\beta_{n+1} \exp(-\mu|\xi|),
\end{aligned}
\label{spectrum:theorem2:6}
\end{equation}
which holds for  all $|\xi| \leq \sqrt{2}\lambda$.
We conclude by induction that  (\ref{spectrum:theorem2:bound2}) holds for all integers
$n \geq 0$.
 
The sequence $\{\psi_n\}$ converges to $\psi$  in $\Lp{1}$ norm 
and so a subsequence of $\psi_n$ converges to $\psi$
pointwise almost everywhere.  
From (\ref{spectrum:theorem2:beta2}) and (\ref{spectrum:theorem2:bound2})
 we conclude that 
\begin{equation}
\left|\psi(\xi)\right| \leq \beta \exp(-\mu|\xi|)
\label{spectrum:theorem2:psibound1}
\end{equation}
for almost all $|\xi| \leq \sqrt{2}\lambda$.  By changing
the values of $\psi$ on  a set of measure $0$ (which does not
affect the value of $R\left[\psi\right]$), we ensure
that (\ref{spectrum:theorem2:psibound1}) holds for all $|\xi| \leq \sqrt{2}\lambda$.
By inserting (\ref{spectrum:theorem2:beta2}) into (\ref{spectrum:theorem2:psibound1})
we obtain
\begin{equation}
\left|\psi(\xi)\right|
\leq
 \left(1 + \frac{2\Gamma}{\lambda} \right)
\ \Gamma \exp(-\mu|\xi|)
\ \ \ \mbox{for all} \ \ \left|\xi\right| \leq \sqrt{2}\lambda,
\label{spectrum:theorem2:psibound15}
\end{equation}
 which is the conclusion (\ref{spectrum:theorem2:bound0}).

We combine (\ref{spectrum:theorem2:psibound1}) with
Theorem~\ref{spectrum:STheorem}  
(the application of which is justified since $\beta <2\Gamma < 2\lambda^2 \mu$)  to conclude that
\begin{equation}
\begin{aligned}
\left|\psi(\xi)\right|
\leq
\exp(-\mu|\xi|)
\left(
\frac{ \beta^2 }{\lambda^2}  \left(\frac{1+\mu|\xi|}{\mu}\right)
\left(
\frac{1}{16\pi} + \frac{1}{2\pi}\exp\left(\frac{\beta}{4\pi \lambda^2 \mu}\right)
\exp\left(\frac{\beta}{4\pi \lambda^2}|\xi|\right)
\right)
+\Gamma
\right)
\end{aligned}
\label{spectrum:theorem2:psibound2}
\end{equation}
for all $\xi \in \mathbb{R}$.  Note the distinction between
(\ref{spectrum:theorem2:bound2}) and (\ref{spectrum:theorem2:psibound2}) is that the former
only holds for $\xi$ in the interval $[-\sqrt{2}\lambda,\sqrt{2}\lambda]$,
while the later holds for all $\xi$ on the real line.
It follows from (\ref{spectrum:theorem2:assumption1}) that
\begin{equation}
\frac{1}{\lambda \mu} < \frac{1}{2}
\ \ \ \mbox{and}\ \ \
\beta < \frac{2\Gamma}{\lambda} < 1.
\end{equation}
We insert these bounds into (\ref{spectrum:theorem2:psibound2})
in order to conclude that
\begin{equation}
\begin{aligned}
\left|\psi(\xi)\right|
&\leq
\Gamma
\exp(-a|\xi|)
\left(
\frac{\beta^2}{\lambda}
\left(
\frac{1}{2}
+
\frac{|\xi|}{\lambda}
\right)
\left(
\frac{1}{16\pi} + \frac{1}{2\pi}\exp\left(\frac{1}{8\pi }\right)
\exp\left(\frac{1}{4\pi \lambda}|\xi|\right)
\right)
+\Gamma
\right)
\end{aligned}
\label{spectrum:theorem2:psibound3}
\end{equation}
for all $\xi \in \mathbb{R}$.
We conclude from  Lemma~\ref{spectrum:exp_lemma} that
\begin{equation}
\exp\left(-\frac{1}{4\pi\lambda}|\xi|\right)
\left(
\frac{1}{2}
+
\frac{|\xi|}{\lambda}
\right)
\leq 
\left(
\frac{1}{2}
+
\frac{4\pi}{\exp(1)}
\right)
\ \ \ \mbox{for all}\ \ \xi\in\mathbb{R}.
\label{spectrum:theorem2:exp1}
\end{equation}
We observe that
\begin{equation}
\exp\left(-\frac{1}{4 \pi \lambda}|\xi|\right)
\left(
\frac{1}{16\pi} + \frac{1}{2\pi}\exp\left(\frac{1}{8\pi }\right)
\exp\left(\frac{1}{4\pi \lambda}|\xi|\right)
\right)
\leq
\left(
\frac{1}{16\pi} + \frac{1}{2\pi}\exp\left(\frac{1}{8\pi }\right)
\right)
\label{spectrum:theorem2:exp2}
\end{equation}
for all $\xi\in\mathbb{R}$,
and that
\begin{equation}
\left(
\frac{1}{2}
+
\frac{4\pi}{\exp(1)}
\right)
\cdot
\left(
\frac{1}{16\pi} + \frac{1}{2\pi}\exp\left(\frac{1}{8\pi }\right)
\right)
<1.
\label{spectrum:theorem2:product}
\end{equation}
We combine  (\ref{spectrum:theorem2:exp1}), (\ref{spectrum:theorem2:exp2}) 
and (\ref{spectrum:theorem2:product})
in order to conclude that
\begin{equation}
\begin{aligned}
&\exp\left(-\frac{1}{ 2 \pi \lambda}\left|\xi\right|\right)
\left(
\frac{\beta^2}{\lambda}
\left(
\frac{1}{2}
+
\frac{|\xi|}{\lambda}
\right)
\left(
\frac{1}{16\pi} + \frac{1}{2\pi}\exp\left(\frac{1}{8\pi }\right)
\exp\left(\frac{1}{4\pi \lambda}|\xi|\right)
\right)
+\Gamma
\right)
&\leq
\frac{\beta^2}{\lambda}+\Gamma
\end{aligned}
\label{spectrum:theorem2:ineq}
\end{equation}
for all $\xi\in\mathbb{R}$.  From (\ref{spectrum:theorem2:beta1})
and  (\ref{spectrum:theorem2:beta2}) we obtain
\begin{equation}
\frac{\beta^2}{\lambda} + \Gamma
=
\frac{\beta^2}{\lambda} + 2 \Gamma - \Gamma
=
2\beta - \Gamma
<
\left( 
1+\frac{4\Gamma}{\lambda} \right) 
\Gamma.
\end{equation}
By inserting (\ref{spectrum:theorem2:ineq}) into (\ref{spectrum:theorem2:psibound2})
we arrive at
\begin{equation}
\left|\psi(\xi)\right|
<
\left(
1+ \frac{4\Gamma}{\lambda} 
\right)\Gamma
\exp\left(-\left(\mu-\frac{1}{2\pi\lambda}\right)\left|\xi\right|\right)
\ \ \ \mbox{for all}\ \ \xi\in\mathbb{R},
\end{equation}
which is  (\ref{spectrum:theorem2:bound}).
\end{proof}

Suppose that $p$ is an element of $\Lp{1}$, and that
there exist positive real numbers $\mu$ and $\Gamma$ such that
\begin{equation}
\left|\widehat{p}(\xi)\right| \leq \Gamma \exp\left(-\mu\left|\xi\right|\right)
\ \ \ \mbox{for all}\ \ \xi \in\mathbb{R}.
\end{equation}
Suppose further that $\psi \in \Lp{1}$ is the solution 
of (\ref{bandlimit:equation}) obtained by invoking Theorem~\ref{convergence:theorem1}.
Then the function $\sigma_b$ defined by the formula
\begin{equation}
\sigma_b(x) = \frac{1}{2\pi} \int_{-\infty}^{\infty} \exp(ix \xi) \psi(\xi)\ d\xi
\label{spectrum:sigma}
\end{equation}
is a solution of the integral equation  (\ref{bandlimit:integral_equation}).
Moreover,  according to Theorem~\ref{spectrum:theorem2},
if $\lambda > 2 \mu^{-1}$ and $\lambda > 2 \Gamma$, then
the Fourier transform of $\sigma_b$ (which is, of course, $\psi$)
decays faster than any polynomial.  In this event, $\sigma_b$
is infinitely differentiable, $\widehat{\sigma_b} \in \Lp{2}$,
and $\sigma_b \in \Lp{2}$.  We record these observations in the following
theorem.

\vskip 1em
\begin{theorem}
Suppose that there exist real numbers $\lambda > 0$, $\Gamma > 0$ and $\mu > 0$ such that
\begin{equation}
\lambda >  2 \max\left\{\Gamma, \frac{1}{\mu}\right\}.
\label{spectrum:theorem3:lambda}
\end{equation}
Suppose also that $p$ is an element of $\Lp{1}$ such that
\begin{equation}
\left|\widehat{p}(\xi)\right| \leq 
\Gamma \exp\left(-\mu|\xi|\right)
\ \ \ \mbox{for all} \ \ \xi \in \mathbb{R}.
\end{equation}
Then there exists a solution $\sigma_b \in \Lp{2} \cap C^\infty\left(\mathbb{R}\right)$
 of the integral equation (\ref{bandlimit:integral_equation}) such that
\begin{equation}
\left|\widehat{\sigma_b}(\xi)\right|\leq
\left(1+\frac{2\Gamma}{\lambda}\right) \Gamma
\exp\left(- \mu|\xi|\right) 
\ \ \ \mbox{for all} \ \ \left|\xi\right| \leq \sqrt{2}\lambda,
\label{spectrum:sigmahat_bound0}
\end{equation}
and
\begin{equation}
\left|\widehat{\sigma_b}(\xi)\right| \leq
\left(
1+ \frac{4\Gamma}{\lambda} 
\right) \Gamma
\exp\left(-\left(\mu-\frac{1}{2\pi\lambda}\right)\left|\xi\right|\right)
\ \ \ \mbox{for all} \ \ \xi\in\mathbb{R}.
\label{spectrum:sigmahat_bound}
\end{equation}
\label{spectrum:theorem3}
\end{theorem}

\label{section:spectrum}
\end{section}

%% file: solution.tex
\begin{section}{Perturbed integral equation}

Suppose that $\sigma_b$ is the function obtained by invoking 
Theorem~\ref{spectrum:theorem3} so that
\begin{equation}
\sigma_b(x) = S\left[T_b\left[\sigma_b\right]\right](x) + p(x).
\label{solution:11}
\end{equation}
We rearrange (\ref{solution:11}) as 
\begin{equation}
\sigma_b(x) = S\left[T\left[\sigma_b\right]\right](x) +  p(x) 
+
\left( S\left[T_b\left[\sigma_b\right]\right](x
) -
S\left[T\left[\sigma_b\right]\right](x) \right)
\end{equation}
 and define $\nu_b$ via the formula
\begin{equation}
\nu_b(x) = 
 S\left[T_b\left[\sigma_b\right]\right](x) -
S\left[T\left[\sigma_b\right]\right](x)
\end{equation}
so that $\sigma_b$ is a solution of the perturbed integral equation
\begin{equation}
\sigma_b(x) = S\left[T\left[\sigma_b\right]\right](x) + p(x) + \nu_b(x).
\end{equation}
From the discussion in Section~\ref{section:overview}, we see that
the phase function arising from $\sigma_b$
approximates solutions of (\ref{introduction:original_equation}) with
accuracy on the order of $\lambda^{-1}\|\nu_b\|_\infty$.
However, it is not immediately apparent that 
$T\left[\sigma_b\right]$ is defined: the integral 
\begin{equation}
\frac{1}{2\lambda} \int_{-\infty}^\infty \sin\left(2\lambda \left|x-y\right|\right) \sigma_b(y)\ dy
\end{equation}
need not converge absolutely,
and without an estimate on the derivative of $\widehat{\sigma_b}$, it is not clear that
\begin{equation}
\frac{\widehat{\sigma_b}(\xi)}{4\lambda^2-\xi^2},
\end{equation}
which is formally the Fourier transform of $T\left[\sigma_b\right]$,
defines a tempered distribution (although, in fact, it does).
It is possible obtain a bound on the derivative of $\widehat{\sigma_b}$
by modifying the argument of Section~\ref{section:spectrum}.  That bound
can be  used
show that $T\left[\sigma_b\right]$ is defined and to estimate the magnitude of $\nu_b$.
We prefer the following, simpler approach to constructing an appropriate
perturbation $\nu$ of the function $p$.

We define the function $\sigma$  via the formula
\begin{equation}
\widehat{\sigma}(\xi) = \widehat{\sigma_b}(\xi) \widehat{b}(\xi),
\end{equation}
where $\widehat{b}(\xi)$ is the function used to define the operator 
$T_b$.   We observe that, unlike $T\left[\sigma_b\right]$,
there is no difficulty in defining $T\left[\sigma\right]$ 
since $\sigma \in \Lp{2}$ and 
the support of $\widehat{\sigma}$  is contained in $(-\sqrt{2}\lambda,\sqrt{2}\lambda)$.
 Moreover, $\OpTb{\sigma_b} = \OpT{\sigma}$ so that
\begin{equation}
\sigma_b(x) = \OpS{\OpT{\sigma}}(x) + p(x)
\ \ \ \mbox{for all}\ \ \ x\in\mathbb{R}.
\label{solution:1}
\end{equation}
Rearranging (\ref{solution:1}), we obtain
\begin{equation}
\sigma(x) = \OpS{\OpT{\sigma}}(x) + 
p(x) + \nu(x)
\ \ \ \mbox{for all}\ \ \ x\in\mathbb{R},
\end{equation}
where $\nu$ is defined the formula
\begin{equation}
\nu(x) = \sigma(x) - \sigma_b(x).
\label{solution:2}
\end{equation}

Using (\ref{spectrum:sigmahat_bound}),
(\ref{spectrum:theorem3:lambda}) and (\ref{solution:2}),
we conclude that under the hypotheses of Theorem~\ref{spectrum:theorem3},
\begin{equation}
\begin{aligned}
\|\nu\|_\infty &
\leq 
\frac{1}{2\pi}
\left\|\widehat{\sigma}-\widehat{\sigma_b}\right\|_1
\\
&\leq
\frac{1}{2\pi} 
\left(1+\frac{4\Gamma}{\lambda}\right)\Gamma
\int_{|\xi|\geq\lambda}
\exp\left(-\left(\mu-\frac{1}{2 \pi \lambda}\right) |\xi|\right)
\ d\xi
\\
&=
\frac{1}{\pi(\mu-\frac{1}{2 \pi \lambda})} \left(1+\frac{4\Gamma}{\lambda}\right)\Gamma
 \exp\left(-\left(\mu-\frac{1}{2 \pi \lambda}\right)\lambda\right)
\\
%% &\leq
%% \exp\left(\frac{1}{2\pi}\right)
%% \frac{1}{\pi\mu(1-\frac{1}{2\pi\lambda\mu})} \left(1+\frac{4\Gamma}{\lambda}\right)
%% \Gamma
%%  \exp\left(-\mu\lambda\right)
%% \\
%% &\leq
%% \exp\left(\frac{1}{2\pi}\right)
%% \frac{1}{\pi(1-\frac{1}{4\pi})} \left(1+\frac{4\Gamma}{\lambda}\right)\frac{\Gamma}{\mu}
%%  \exp\left(-\mu\lambda\right)
%% \\
&<
\frac{\Gamma}{2\mu}
 \left(1+\frac{4\Gamma}{\lambda}\right)
 \exp\left(-\mu\lambda\right).
\end{aligned}
\label{solution:3}
\end{equation}
By combining Theorem~\ref{spectrum:theorem3}
with (\ref{solution:3}) we arrive at the following theorem.
\vskip 1em
\begin{theorem}
Suppose that  $q \in C^\infty\left(\mathbb{R}\right)$ is strictly positive,
and that $x(t)$ is defined by the formula
\begin{equation}
x(t) = \int_0^t \sqrt{q(u)}\ du.
\label{solution:x}
\end{equation}
Suppose also that $p(x)$ is defined via the formula 
\begin{equation}
p(x) = 2\{t,x\};
\end{equation}
that is, $p(x)$ is twice the Schwarzian derivative of the
variable $t$ with respect to the variable $x$ defined via
 (\ref{solution:x}).  Suppose furthermore that there exist positive real numbers
$\lambda$, $\Gamma$ and $\mu$ such that
\begin{equation}
\lambda >  2\max\left\{\frac{1}{\mu},\Gamma\right\}
\end{equation}
and
\begin{equation}
\left|\widehat{p}(\xi)\right| \leq 
\Gamma \exp\left(-\mu\left|\xi\right|\right)
\ \ \ \mbox{for all}\ \ \xi\in\mathbb{R}.
\end{equation}
Then there exist functions $\nu$ and $\sigma$ 
in $\Lp{2} \cap C^{\infty}\left(\mathbb{R}\right)$
such that
%% such that $\widehat{\sigma}$ and  $\widehat{\nu}$ are elements
%% of $\Lp{2} \cap C\left(\mathbb{R}\right)$,
%% %% $\widehat{\sigma}$ is an element of $\Lp{1}$,
$\sigma$ is a solution of the nonlinear integral equation
\begin{equation}
\sigma(x) = S\left[T\left[\sigma\right]\right](x) + p(x) + \nu(x),
\label{theorem1:inteq}
\end{equation}
\begin{equation}
\left|\widehat{\sigma}(\xi)\right| \leq
\left(1+\frac{2\Gamma}{\lambda}\right) \Gamma
\exp\left(- \mu|\xi|\right) 
\ \ \ \mbox{for all}\ \ \left|\xi\right| < \sqrt{2}\lambda,
\label{theorem1:sigma1}
\end{equation}
\begin{equation}
\widehat{\sigma}(\xi) = 0 
\ \ \ \mbox{for all}\ \ \left|\xi\right| \geq \sqrt{2}\lambda,
\end{equation}
and
\begin{equation}
\|\nu\|_\infty <
\frac{\Gamma}{2\mu}
 \left(1+\frac{4\Gamma}{\lambda}\right)
 \exp\left(-\mu\lambda\right).
\label{theorem1:nu}
\end{equation}
\label{solution:theorem1}
\end{theorem}

%% The solution $\sigma$ of the integral equation ()
%%  Theorem~\ref{solution:theorem1}
%% is obtained by first applying the Fourier transform to the
%% nonlinear integral equation (\ref{inteq:integral_equation}) in order
%% to obtain the equation .  
%% sequence $\left\{\psi_n\right\}$ of fixed point iterates for (\ref{})
%% is constructed

The function $\sigma_b$ is obtained as the inverse Fourier transform
of the limit $\psi$ of a sequence of fixed point iterates
$\{\psi_n\}$ for the equation (\ref{bandlimit:equation}).
As a consequence of  Theorem~\ref{preliminaries:convolution:theorem2},
the functions $\psi_n$ are elements of the space $\Sr$ of
Schwartz functions if $\widehat{p}$ is an element of $\Sr$.
Thus $\sigma_b$ is the limit in $\Lp{1}$ of a sequence of Schwartz
functions.  The following proof of the principal
result of this paper, Theorem~\ref{main_theorem},
proceeds by  approximating $\sigma_b$ using 
the inverse Fourier transform of an appropriately chosen $\psi_n$.

{\it Proof of Theorem~\ref{main_theorem}.}
We denote by   $\left\{\psi_n\right\}$ the sequence of fixed point
iterates for (\ref{bandlimit:equation}) generated by the function
$\widehat{p}$.  That is, $\psi_0$ is defined by the formula
\begin{equation}
\psi_0(\xi) = \widehat{p}(\xi)
\label{solution2:1}
\end{equation}
and, for each integer $n=0,1,2,\ldots$, $\psi_{n+1}$ is defined via
the formula
\begin{equation}
\psi_{n+1}(\xi) = R\left[\psi_n\right](\xi),
\label{solution2:2}
\end{equation}
where $R$ is as in (\ref{bandlimit:definition_of_R}).
According to  Theorem~\ref{convergence:theorem1}, the sequence
$\{\psi_n\}$ converges in $\Lp{1}$ to a function $\psi$ such
that 
\begin{equation}
\psi(\xi) = R\left[\psi\right](\xi)\ \ \ \mbox{for all}\  \ \xi \in\mathbb{R}.
\label{solution2:3}
\end{equation}
Moreover, as a consequence of Theorem~\ref{preliminaries:convolution:theorem2}
and the assumption that $\widehat{p}\in\Sr$,
each of the $\psi_n$ is an element of $\Sr$. We denote the inverse Fourier transform
of $\psi$ by $\tilde{\sigma_b}$.
According to Theorem~\ref{spectrum:theorem3}, $\tilde{\sigma_b}$ is a solution
of the nonlinear integral equation
\begin{equation}
\tilde{\sigma_b}(x) = S\left[T_b\left[\tilde{\sigma_b}\right]\right](x) + p(x).
\label{solution2:3.5}
\end{equation}
We now define $\tilde{\sigma}$ via the formula
\begin{equation}
\widehat{\tilde{\sigma}}(\xi) = \psi(\xi) \widehat{b}(\xi),
\label{solution2:4}
\end{equation}
where $\widehat{b}$ is given in (\ref{preliminaries:bump:b1}),
so that
\begin{equation}
T_b\left[\tilde{\sigma_b}\right](\xi) = T\left[\tilde{\sigma}\right](\xi).
\label{solution2:4.5}
\end{equation}
And we  define $\tilde{\nu}$ via the formula
\begin{equation}
\tilde{\nu}(\xi) =  \tilde{\sigma}(\xi) - \tilde{\sigma_b}(\xi).
\label{solution2:5}
\end{equation}
From (\ref{solution2:3.5}), (\ref{solution2:4.5}) and (\ref{solution2:5}),
we conclude that
\begin{equation}
\tilde{\sigma}(x) = S\left[T\left[\tilde{\sigma}\right]\right](x) + p(x) + \tilde{\nu}(x)
\label{solution2:6}
\end{equation}
for all $x \in \mathbb{R}$.    According to Theorem~\ref{solution:theorem1},
\begin{equation}
\|\tilde{\nu}\|_\infty <
\frac{\Gamma}{2\mu}
 \left(1+\frac{4\Gamma}{\lambda}\right)
 \exp\left(-\mu\lambda\right).
\label{solution2:7}
\end{equation}

For each integer $n=0,1,2,\ldots$, we denote by  $\sigma_n$ the 
inverse Fourier transform of the function
\begin{equation}
 \psi_n(\xi)\widehat{b}(\xi).
\label{solution2:8}
\end{equation}
The function $\psi_n$ is in $\Sr$ and $\widehat{b}$ is an element
of $C_c^\infty\left(\mathbb{R}\right)$, so the $\sigma_n$ are contained
in $\Sr$.
We combine (\ref{solution2:4}) and (\ref{solution2:8}) in order to obtain
\begin{equation}
\left\|\sigma_n - \tilde{\sigma}\right\|_\infty
%% \leq 
%% \frac{1}{2\pi}
%% \left\|\widehat{\sigma_n} - \widehat{\tilde{\sigma}}\right\|_1
\leq
\frac{1}{2\pi}
\int_{-\infty}^\infty
\left|
\widehat{b}(\xi)\psi_n(\xi) - \widehat{b}(\xi)\psi(\xi)
\right|\ d\xi
\leq
\frac{1}{2\pi}
\int_{-\infty}^\infty
\left|
\psi_n(\xi) -\psi(\xi)
\right|\ d\xi
\to 0 \ \ \mbox{as}\ \ n\to\infty.
\label{solution2:9}
\end{equation}
%
%% for all nonnegative integers $n$.  From (\ref{solution2:9}),
%%  we conclude that 
%% %
%% \begin{equation}
%% \left\|\sigma_n - \tilde{\sigma}\right\|_\infty \to 0 \ \ \mbox{as} \ \ n\to\infty.
%% \label{solution2:10}
%% \end{equation}
%
Since $\widehat{b}$ is supported on $(-\sqrt{2}\lambda,\sqrt{2}\lambda)$,
\begin{equation}
\begin{aligned}
\left\|T\left[\sigma_n\right]-T\left[\tilde{\sigma}\right]\right\|_\infty
&\leq 
\frac{1}{2\pi}
\int_{-\infty}^\infty
\left|
\frac{\widehat{b}(\xi)\psi_n(\xi) - \widehat{b}(\xi)\psi(\xi)}{4\lambda^2-\xi^2}
\right|\ d\xi
\\
&\leq 
\frac{1}{4\pi\lambda^2}
\int_{-\infty}^\infty
\left|
\widehat{b}(\xi)\psi_n(\xi) - \widehat{b}(\xi)\psi(\xi)
\right|\ d\xi
\\
&\leq
\frac{1}{4\pi\lambda^2}
\int_{-\infty}^\infty
\left|
\psi_n(\xi) - \psi(\xi)
\right|\ d\xi,
\label{solution2:11}
\end{aligned}
\end{equation}
from which we conclude that $T\left[\sigma_n\right]$ converges to
$T\left[\tilde{\sigma}\right]$ in $\Lp{\infty}$.  Similarly,
if we denote the derivative of the function $T\left[f\right](x)$
with respect to $x$  by $T\left[f\right]'(x)$, then
\begin{equation}
\begin{aligned}
\left\|T\left[\sigma_n\right]'-T\left[\tilde{\sigma}\right]'\right\|_\infty
&\leq 
\frac{1}{2\pi}
\int_{-\infty}^\infty
\left|
\frac{-i\xi\widehat{b}(\xi)\psi_n(\xi) + i\xi\widehat{b}(\xi)\psi(\xi)}{4\lambda^2-\xi^2}
\right|\ d\xi
\\
&\leq 
\frac{1}{2\sqrt{2}\pi  \lambda}
\int_{-\infty}^\infty
\left|
\psi_n(\xi) - \psi(\xi)
\right|\ d\xi,
\label{solution2:12}
\end{aligned}
\end{equation}
from which we conclude that the derivative of $T\left[\sigma_n\right]$
converges to the derivative of $T\left[\tilde{\sigma}\right]$ in $\Lp{\infty}$.
We combine these observations with the definition (\ref{inteq:definition_of_S})
in order to conclude that
\begin{equation}
\left\|S\left[T\left[\sigma_n\right]\right] - S\left[T\left[\tilde{\sigma}\right]\right]
\right\|_\infty \to 0 \ \ \mbox{as}\ \ n\to\infty.
\label{solution2:13}
\end{equation}

We rearrange (\ref{solution2:6}) as
\begin{equation}
\sigma_n(x) = S\left[T\left[\sigma_n\right]\right](x) + p (x)
+ \nu_n(x),
\label{solution2:14}
\end{equation}
where $\nu_n$ is defined via the formula
\begin{equation}
\nu_n(x) = 
\tilde{\nu}(x)
+ \left(S\left[T\left[\tilde{\sigma}\right]\right](x) -
S\left[T\left[\sigma_n\right]\right]\right)(x)
+
\left(\sigma_n(x) - \tilde{\sigma}(x)\right).
\label{solution2:15}
\end{equation}
We combine (\ref{solution2:15}) with 
(\ref{solution2:9}) and (\ref{solution2:13}) in order
to conclude that
\begin{equation}
\left\|\nu_n - \tilde{\nu}\right\|_\infty \to 0 \ \ \mbox{as} \ \ n\to \infty.
\label{solution2:16}
\end{equation}
Together (\ref{solution2:16}) and   (\ref{solution2:7})
imply
\begin{equation}
\left\|\nu_n\right\|_\infty 
\leq 
\frac{\Gamma}{2\mu}
 \left(1+\frac{4\Gamma}{\lambda}\right)
 \exp\left(-\mu\lambda\right)
\label{solution2:17}
\end{equation}
when $n$ is sufficiently large.  Note that the inequality (\ref{solution2:7})
is strict.

We have already established that $\psi_n$ and
$\sigma_n$  are elements of $\Sr$ for all nonnegative integers $n$.
  Now we observe that
\begin{equation}
\widehat{T\left[\sigma_n\right]}(\xi) 
=
\frac{\widehat{\sigma_n}(\xi)}{4\lambda^2-\xi^2}
=
\frac{\widehat{\psi_n}(\xi) \widehat{b}(\xi)}{4\lambda^2-\xi^2}.
\label{solution2:333}
\end{equation}
Since the support of $\widehat{b}$ is contained in $(-\sqrt{2}\lambda,\sqrt{2}\lambda)$
and $\psi_n$ is an element of $\Sr$, we conclude from (\ref{solution2:333})
that the Fourier transform of $T\left[\sigma_n\right]$ --- and hence
$T\left[\sigma_n\right]$ --- is an element of $\Sr$
for each nonnegative integer $n$.  Next, we combine this observation
with Theorem~\ref{preliminaries:convolution:theorem1}  in order to conclude 
that  $S\left[T\left[\psi_n\right]\right] \in \Sr$ for all nonnegative
integers $n$.  We rearrange (\ref{solution2:14}) as
\begin{equation}
\nu_n(x)
 = \sigma_n(x) - S\left[T\left[\sigma_n\right]\right](x)  - p (x)
\label{solution2:334}
\end{equation}
and observe that all of the functions appearing on the right-hand side
of (\ref{solution2:334}) are elements of $\Sr$.  We conclude that
$\nu_n \in \Sr$ for all nonnegative integers $n$.

It follows from Formulas~ (\ref{spectrum:theorem2:beta2})
and (\ref{spectrum:theorem2:bound2}), which appear
 in the proof of Theorem~\ref{spectrum:theorem3}, 
 that 
\begin{equation}
\left|\psi_n(\xi)\right|
\leq
\left(1+\frac{2\Gamma}{\lambda}\right) \Gamma
\exp\left(- \mu|\xi|\right) 
\label{solution2:19}
\end{equation}
for all $|\xi| \leq \sqrt{2}\lambda$ and all nonnegative integers $n$.
We conclude from (\ref{solution2:14})
(\ref{solution2:17}) and (\ref{solution2:19}),
and our observation that $\sigma_n$ and $\nu_n$ are elements
of $\Sr$ for all nonnegative integers $n$
that we obtain Theorem~\ref{main_theorem}
by letting  $\sigma = \sigma_n$ and $\nu = \nu_n$ 
for a sufficiently large $n$. 
\hfill
\fbox{\phantom{\rule{.2ex}{.2ex}}}

\label{section:solution}
\end{section}

%% file: acknowledgements.tex
\begin{section}{Acknowledgments}
James Bremer was supported in part by a fellowship from the Alfred P. Sloan
Foundation and by National Science Foundation grant DMS-1418723.
Vladimir Rokhlin was supported in part by 
Office of Naval Research contracts
ONR N00014-10-1-0570 and ONR N00014-11-1-0718, and by the Air Force Office of Scientific
Research under contract AFOSR FA9550-09-1-0241.

\end{section}